\documentclass[12pt]{amsart}
\usepackage{amssymb,amsmath,tabularx,mathrsfs,mathbbol,yfonts,upgreek}
\usepackage{amsthm,verbatim,comment}
\usepackage[bookmarks=true]{hyperref}
\usepackage{pstricks,pst-node,pst-plot}
\usepackage{geometry}
\usepackage{stmaryrd}
\usepackage{paralist}
\usepackage[all]{xy}
\usepackage{array}
\usepackage{url}
\usepackage{longtable}
\numberwithin{equation}{section}

\DeclareSymbolFontAlphabet{\mathbb}{AMSb}
\DeclareSymbolFontAlphabet{\mathbbl}{bbold}

\newtheorem{thm}{Theorem}[section]

\newtheorem{lem}[thm]{Lemma}

\newtheorem{cor}[thm]{Corollary}

\theoremstyle{definition}
\newtheorem{defn}[thm]{Definition}
\newtheorem{nota}[thm]{Notation}
\newtheorem{eg}[thm]{Example}
\newtheorem{rem}[thm]{Remark}
\theoremstyle{remarks}

\newtheorem*{rem*}{Remarks}
\newtheorem{ques}[thm]{Question}

\newtheoremstyle{case}{}{}{}{}{}{:}{ }{}
\theoremstyle{case}

\newcommand{\F}{\mathbb{F}}

\title[A note on modular Terwilliger algebras of association schemes]{A note on modular Terwilliger algebras of association schemes}
\begin{document}
\author{Yu Jiang}
\address[Y. Jiang]{Division of Mathematical Sciences, Nanyang Technological University, SPMS-MAS-05-34, 21 Nanyang Link, Singapore 637371.}
\email[Y. Jiang]{jian0089@e.ntu.edu.sg}


\begin{abstract}
Let $p$ denote a prime number. In this note, we focus on the modular Terwilliger algebras of association schemes defined in \cite{Han1}. We define the primary module of a modular Terwilliger algebra of an association scheme and determine all its composition factors up to isomorphism. We then characterize the $p'$-valenced association schemes by some properties of their modular Terwilliger algebras. The corollaries about the modular Terwilliger algebras of association schemes are given.
\vspace{-1.5em}
\end{abstract}
\maketitle
\noindent{\textbf{Keywords.} Association scheme; Modular Terwilliger algebra; $p'$-valenced scheme}
\section{Introduction}
Fix a nonempty finite set $X$ and call an association scheme on $X$ a scheme.

The subconstituent algebras of schemes, introduced by Terwilliger in \cite{T}, are now known as the Terwilliger algebras of schemes. These algebras are finite-dimensional semisimple $\mathbb{C}$-algebras.
Let $\F$ denote a fixed field of positive characteristic $p$. In \cite{Han1}, Hanaki defined the Terwilliger algebras of schemes over a commutative unital ring. Hanaki also called the Terwilliger algebras of schemes over $\F$ the modular Terwilliger algebras of schemes and studied their basic properties. By the definition of a modular Terwilliger algebra of a scheme, the adjacency algebra of a scheme over $\F$ is contained in every modular Terwilliger algebra of this scheme. Hence the modular Terwilliger algebras of schemes may contain more combinatorial information than the adjacency algebras of schemes over $\F$. Therefore it is necessary to study the modular Terwilliger algebras of schemes.

The primary module of a Terwilliger algebra of a scheme $S$ is an irreducible module of this algebra. It closely relates to the structure of $S$ (see \cite{E}, \cite{T}, \cite{T1}, \cite{T2}). We define the primary module for a modular Terwilliger algebra of $S$ (see Definition \ref{D;primarymodule}). So it is natural to study the primary module of a modular Terwilliger algebra of $S$. As the first main result of this note, in Theorem \ref{T;primarymodule}, we determine all composition factors of the primary module of a modular Terwilliger algebra of $S$ up to isomorphism.

The $p'$-valenced schemes enjoy many algebraic properties. For example, according to \cite[Theorem 3.4]{Han1}, a modular Terwilliger algebra of $S$ is semisimple only if $S$ is a $p'$-valenced scheme. This listed example motivates us to characterize the $p'$-valenced schemes by the properties of their modular Terwilliger algebras. As the second main result of this note, in Theorem \ref{T;characterization}, we characterize the $p'$-valenced schemes by some properties of their modular Terwilliger algebras.

The organization of this note is as follows. The basic notation and preliminaries are given in Section $2$. Two main results are proved in Sections $3$ and $4$, respectively.
\section{Basic notation and preliminaries}
For a general background on association schemes, one may refer to \cite{EI}, \cite{Z2}, or \cite{Z3}.
\subsection{General conventions}
Throughout the present note, fix a field $\F$ of positive characteristic $p$ and a nonempty finite set $X$. Let $\mathbb{N}$ be the set of all natural numbers. Set $\mathbb{N}_0=\mathbb{N}\cup\{0\}$. If $a,b\in \mathbb{N}_0$, put $[a,b]=\{c\in \mathbb{N}_0: a\leq c\leq b\}$. For a nonempty set $Y$ of an $\F$-linear space, write $\langle Y\rangle_\F$ for the $\F$-linear space generated by $Y$. By convention, $\langle \varnothing\rangle_\F$ is the zero space. The addition, the multiplication, and the scalar multiplication of matrices in this note are the usual matrix operations. A scheme means an association scheme on $X$. All modules are finitely generated left modules.
\subsection{Schemes}
Let $S=\{R_0, R_1,\ldots, R_d\}$ be a partition of the Cartesian product $X\times X$. Then $S$ is a scheme of class $d$ if the following conditions hold:
\begin{enumerate}[(i)]
\item $R_0=\{(x,x): x\in X\}$;
\item For any $i\in [0,d]$, there is $i'\in [0,d]$ such that $\{(x,y): (y,x)\in R_i\}=R_{i'}\in S$;
\item For any $i,j,\ell\in [0,d]$ and $(x,y), (\tilde{x},\tilde{y})\in R_\ell$, the following equality holds: $|\{z\in X: (x,z)\in R_i,\ (z,y)\in R_j\}|=|\{z\in X: (\tilde{x},z)\in R_i,\ (z,\tilde{y})\in R_j\}|.$
\end{enumerate}

Throughout the whole note, $S=\{R_0, R_1,\ldots, R_d\}$ is a fixed scheme of class $d$. By (iii), for any $i,j,\ell\in [0,d]$ and $(x,y)\in R_\ell$, there exists a constant $p_{ij}^\ell\in \mathbb{N}_0$ such that $p_{ij}^\ell=|\{z\in X: (x,z)\in R_i,\ (z,y)\in R_j\}|$.
Let $i\in [0,d]$ and $x,y\in X$. Set $xR_i=\{z\in X: (x,z)\in R_i\}$. Write $k_i$ for $p_{ii'}^{0}$. The number $k_i$ is called the valency of $R_i$. Observe that $|xR_i|=|yR_i|=k_i$. Hence $k_i>0$ since $R_i\neq\varnothing$. If $n\in \mathbb{N}_0$, set $S_n=\{j\in [0,d]: p^n\mid k_j\ \text{and}\ p^{n+1}\nmid k_j\}$. Call $S$ a $p'$-valenced scheme if $S_0=[0,d]$. Put $O_\vartheta(S)=\{R_j\in S: k_j=1\}$. Note that $R_0\in O_\vartheta(S)$. We list a lemma as follows.
\begin{lem}\label{L;trigular}\cite[Proposition 2.2 (vi)]{EI}
Let $i,j,\ell\in [0,d]$. Then $k_\ell p_{ij}^{\ell}=k_i p_{\ell j'}^{i}=k_jp_{i'\ell}^{j}$.
\end{lem}
\subsection{Algebras}
Let $\mathbb{Z}$ be the integer ring and $\F_p$ be the prime subfield of $\F$. Given $a\in \mathbb{Z}$, let $\overline{a}$ be the image of $a$ under the unital ring homomorphism from $\mathbb{Z}$ to $\F_p$.

Let $A$ be a finite-dimensional associative unital $\F$-algebra and $B$ be a two-sided ideal of $A$. Call $B$ a nilpotent ideal of $A$ if there is $m\in \mathbb{N}$ such that the product of any $m$ elements of $B$ is the zero element of $A$. Let $\mathrm{Rad}(A)$ be the Jacobson radical of $A$. Recall that $\mathrm{Rad}(A)$ is the sum of all nilpotent two-sided ideals of $A$. Let $U$ be a nonzero $A$-module and $V, W$ be $A$-modules. Call $U$ an irreducible $A$-module if $U$ has no nonzero proper $A$-submodule. Call $U$ an indecomposable $A$-module if $U$ is not a direct sum of two nonzero $A$-submodules. Let $\mathrm{Ann}_A(V)=\{\hat{a}\in A: \hat{a}\hat{v}=0\ \forall\ \hat{v}\in V\}$. Observe that $\mathrm{Ann}_A(V)$ is a two-sided ideal of $A$. If $V$ is an irreducible $A$-module, it is known that $\mathrm{Rad}(A)\subseteq \mathrm{Ann}_A(V)$. If $W$ is an $A$-submodule of $V$, $V/W$ denotes the quotient $A$-module of $V$ with respect to $W$. A composition series of $U$ of length $n$ is an $A$-submodule series $U_n\subset U_{n-1}\subset\cdots\subset U_{1}\subset U_0=U$ of $U$, where $U_n$ is the zero module and $U_{q-1}/U_{q}$ is an irreducible $A$-module for every $q\in [1,n]$. Call $U_{q-1}/U_{q}$ a composition factor of $U$ for every $q\in [1,n]$. By the Jordan-H\"{o}lder Theorem, all composition factors of $U$ are independent of the choice of composition series of $U$ up to permutation and isomorphism. So the length of a composition series of $U$ is an invariant of $U$. Call this invariant the composition length of $U$.
\subsection{Modular Terwilliger algebras of schemes}\label{S;primarymodule}
Let $\F^X$ be the $\F$-linear space of $\F$-column vectors whose coordinates are labeled by the elements of $X$. Let $\mathbf{1}$ and $\mathbf{0}$ be the all-one column vector and the all-zero column vector in $\F^X$, respectively. Let $M_X(\F)$ be the full matrix algebra of $\F$-square matrices whose rows and columns are labeled by the elements of $X$. Let $I$, $J$, $O$ denote the identity matrix, the all-one matrix, and the all-zero matrix in $M_X(\F)$, respectively. If $Z\in M_X(\F)$, let $Z^t$ denote the transpose of $Z$. Let $y,z\in X$ and $E_{yz}$ denote the $\{0,1\}$-matrix in $M_X(\F)$ whose unique nonzero entry is the $(y,z)$-entry.

Let $i,j\in [0,d]$. The adjacency $\F$-matrix with respect to $R_i$, denoted by $A_i$, is the matrix $\sum_{(\tilde{x},\tilde{y})\in R_i}E_{\tilde{x}\tilde{y}}$. The dual $\F$-idempotent with respect to $y$ and $R_i$, denoted by $E_i^*(y)$, is the matrix $\sum_{\tilde{z}\in yR_i}E_{\tilde{z}\tilde{z}}$. Let $\delta_{\alpha\beta}$ denote the Kronecker delta of integers $\alpha$ and $\beta$ whose values are in $\F$. Notice that
\begin{align}\label{Eq;preliminary1}
A_i^t=A_{i'}\ \text{and}\ E_i^*(y)^t=E_i^*(y),
\end{align}
\begin{align}\label{Eq;preliminary2}
A_0=I=\sum_{\ell=0}^dE_\ell^*(y)\ \text{and}\ J=\sum_{\ell=0}^dA_\ell,
\end{align}
\begin{align}\label{Eq;preliminary3}
E_i^*(y)E_j^*(y)=\delta_{ij}E_i^*(y),
\end{align}
\begin{align}\label{Eq;preliminary4}
E_i^*(y)JE_j^*(y)\neq O,
\end{align}
\begin{align}\label{Eq;preliminary5}
JE_i^*(y)\mathbf{1}=\overline{k_i}\mathbf{1}\ (\text{in particular},\ JE_i^*(y)J=\overline{k_i}J).
\end{align}

Let $T(y)$ be the modular Terwilliger algebra of $S$ with respect to $y$. It is the unital $\F$-subalgebra of $M_X(\F)$ generated by $A_0, A_1,\ldots, A_d, E_0^*(y), E_1^*(y),\ldots, E_d^*(y)$. Recall that the algebraic structures of $T(y)$ and $\mathrm{Rad}(T(y))$ may depend on $y$. For example, it is possible that $\dim_\F T(y)\neq\dim_\F T(z)$ and $\dim_\F\mathrm{Rad}(T(y))\neq\dim_\F\mathrm{Rad}(T(z))$ at the same time (see \cite[5.1]{Han1}). From now on, fix $x\in X$. Set $T=T(x)$ and $E_\ell^*=E_\ell^*(x)$ for every $\ell\in [0,d]$. If $n\in \mathbb{N}_0$ and $i_m, j_m, \ell_m\in [0,d]$ for every $m\in [0,n]$, put
\begin{align}\label{Eq;preliminary6}
\prod_{m=0}^n(E_{i_m}^*A_{j_m}E_{\ell_m}^*)=E_{i_0}^*A_{j_0}E_{\ell_0}^*E_{i_1}^*A_{j_1}E_{\ell_1}^*\cdots E_{i_n}^*A_{j_n}E_{\ell_n}^*\in T\ \text{and}
\end{align}
\[
T_n=\{\prod_{m=0}^n(E_{r_m}^*A_{s_m}E_{t_m}^*): r_m, s_m, t_m\in [0,d]\ \text{for every $m\in [0,n]$}\}.
\]

We end the whole section with the following lemma.
\begin{lem}\label{L;generalresults}
Let $i,j,\ell\in [0,d]$.
\begin{enumerate}[(i)]
\item [\em (1)]\cite[Lemma 3.2]{Han1} $E_i^*A_jE_\ell^*\mathbf{1}=\overline{p_{\ell j'}^i}E_i^*\mathbf{1}$. In particular, $E_i^*A_jE_\ell^*J=\overline{p_{\ell j'}^i}E_i^*J$.
\item [\em (ii)] \cite[Lemma 3.2]{Jiang} If $E_i^*A_jE_\ell^*\neq O$ and $\min\{k_i,k_\ell\}=1$, then $E_i^*A_jE_\ell^*=E_i^*JE_\ell^*$.
\item [\em (iii)]Every element of $T$ is a finite $\F$-linear combination of elements of $\bigcup_{m\in \mathbb{N}_0}T_m$.
\end{enumerate}
\end{lem}
\begin{proof}
For (iii), $A_0, A_1,\ldots, A_d, E_0^*, E_1^*,\ldots, E_d^*\in \langle\bigcup_{m\in \mathbb{N}_0}T_m\rangle_\F$ by \eqref{Eq;preliminary2}, \eqref{Eq;preliminary3}, \eqref{Eq;preliminary6}. Notice that $\langle\bigcup_{m\in \mathbb{N}_0}T_m\rangle_\F$ is a unital $\F$-subalgebra of $T$ by \eqref{Eq;preliminary6}. So $T=\langle\bigcup_{m\in \mathbb{N}_0}T_m\rangle_\F$ by the definition of $T$. (iii) thus follows as $T=\langle\bigcup_{m\in \mathbb{N}_0}T_m\rangle_\F$ and $\dim_\F T\in \mathbb{N}$.
\end{proof}
\section{Primary modules of modular Terwilliger algebras of schemes}
In this section, we define the primary module of $T$ and study its basic properties. In particular, we  determine all its composition factors up to isomorphism. For our purpose, notice that $M_X(\F)$ acts on $\F^X$ by left multiplication. We first list a lemma.
\begin{lem}\label{L;primemodule}
The following statements hold:
\begin{enumerate}[(i)]
\item [\em (i)] $\{E_i^*\mathbf{1}: i\in [0,d]\}$ is an $\F$-linearly independent subset of $\F^X$ of cardinality $d+1$.
\item [\em (ii)] $\langle\{E_i^*\mathbf{1}: i\in [0,d]\}\rangle_\F$ is a $T$-module under the left multiplication action of $T$.
\item [\em (iii)] Assume that $n\in \mathbb{N}$. Then $\langle\{E_i^*\mathbf{1}: i\in [0,d],\ p^n\mid k_i\}\rangle_\F$ is a $T$-module under the left multiplication action of $T$.
\end{enumerate}
\end{lem}
\begin{proof}
For (i), for every $h\in [0,d]$, notice that $E_h^*\mathbf{1}\neq \mathbf{0}$ by the definitions of $E_h^*$ and $\mathbf{1}$. Suppose that there exist $c_0, c_1,\ldots, c_d\in \F$ such that $(\bigcup_{i=0}^d\{c_i\})\cap(\F\setminus\{0\})\neq \varnothing$ and
$\sum_{i=0}^dc_iE_i^*\mathbf{1}=\mathbf{0}$. So there is $j\in [0,d]$ such that $c_j\neq 0$. By \eqref{Eq;preliminary3}, observe that
$c_jE_j^*\mathbf{1}=E_j^*(\sum_{i=0}^dc_iE_i^*\mathbf{1})=E_j^*\mathbf{0}=\mathbf{0}$. Since $E_j^*\mathbf{1}\neq \mathbf{0}$, we thus have $c_j=0$, which contradicts the inequality $c_j\neq 0$. So $\{E_i^*\mathbf{1}: i\in [0,d]\}$ is an $\F$-linearly independent subset of $\F^X$. We also note that $|\{E_i^*\mathbf{1}: i\in [0,d]\}|=d+1$ by \eqref{Eq;preliminary3}. (i) thus follows.

Let $a, b, c\in [0,d]$. For (ii), by \eqref{Eq;preliminary3} and Lemma \ref{L;generalresults} (i), $E_a^*A_bE_c^*E_h^*\mathbf{1}=\delta_{ch}\overline{p_{hb'}^a}E_a^*\mathbf{1}$ for every $h\in [0,d]$. In particular, notice that $E_a^*A_bE_c^*E_h^*\mathbf{1}\in \langle\{E_i^*\mathbf{1}: i\in [0,d]\}\rangle_\F$ for every $h\in [0,d]$. Since $\langle\{E_i^*\mathbf{1}: i\in [0,d]\}\rangle_\F$ is an $\F$-linear space and $a,b,c$ are chosen from $[0,d]$ arbitrarily, (ii) thus follows from Lemma \ref{L;generalresults} (iii) and \eqref{Eq;preliminary6}.

For (iii), notice that (iii) is trivial if $\langle\{E_i^*\mathbf{1}: i\in [0,d],\ p^n\mid k_i\}\rangle_\F=\{\mathbf{0}\}$. We thus assume further that $\langle\{E_i^*\mathbf{1}: i\in [0,d],\ p^n\mid k_i\}\rangle_\F\neq\{\mathbf{0}\}$. For every $h\in [0,d]$ and $p^n\mid k_h$, we claim that $E_a^*A_bE_c^*E_h^*\mathbf{1}\in \langle\{E_i^*\mathbf{1}: i\in [0,d],\ p^n\mid k_i\}\rangle_\F$. Suppose that there is $\ell\in [0,d]$ such that $E_a^*A_bE_c^*E_\ell^*\mathbf{1}\notin \langle\{E_i^*\mathbf{1}: i\in [0,d],\ p^n\mid k_i\}\rangle_\F$ and $p^n\mid k_\ell$. So $E_a^*A_bE_c^*E_\ell^*\mathbf{1}\neq \mathbf{0}$. So $E_a^*A_bE_c^*E_\ell^*\mathbf{1}=\overline{p_{\ell b'}^a}E_a^*\mathbf{1}\notin \langle\{E_i^*\mathbf{1}: i\in [0,d],\ p^n\mid k_i\}\rangle_\F$ by \eqref{Eq;preliminary3} and Lemma \ref{L;generalresults} (i). We thus have $p\nmid p_{\ell b'}^a$ and $p^n\nmid k_a$. Since $p^n\mid k_\ell$, $p^n\nmid k_a$, and $k_ap_{\ell b'}^a=k_\ell p_{ab}^\ell$ by Lemma \ref{L;trigular}, observe that $p\mid p_{\ell b'}^a$, which contradicts the fact $p\nmid p_{\ell b'}^a$. The desired claim follows. As  $\langle\{E_i^*\mathbf{1}: i\in [0,d],\ p^n\mid k_i\}\rangle_\F$ is an $\F$-linear space and $a,b,c$ are chosen from $[0,d]$ arbitrarily, (iii) is proved by combining Lemma \ref{L;generalresults} (iii), \eqref{Eq;preliminary6}, and the proven claim.
\end{proof}
We are now ready to define the primary module of $T$.
\begin{defn}\label{D;primarymodule}
The $T$-module in Lemma \ref{L;primemodule} (ii) is similar to the primary module of a Terwilliger algebra of $S$ (see \cite[Lemma 3.6]{T}). Call the $T$-module in Lemma \ref{L;primemodule} (ii) the primary module of $T$ and denote it by $W_0$. Let $n\in \mathbb{N}$. Let $W_n$ denote the $T$-module in Lemma \ref{L;primemodule} (iii). For every $m\in \mathbb{N}_0$, note that $W_{m+1}$ is a $T$-submodule of $W_m$. Notice that $W_1\subset W_0$. Notice that $\dim_\F W_0=d+1$ by Lemma \ref{L;primemodule} (i).
\end{defn}
\begin{lem}\label{L;submoduleseries}
Let $n\in \mathbb{N}_0$.
\begin{enumerate}[(i)]
\item [\em (i)] If $S_m=\varnothing$ for every $n<m\in \mathbb{N}$, then $W_m=\{\mathbf{0}\}$ for every $n<m\in \mathbb{N}$.
\item [\em (ii)] $W_n/W_{n+1}$ has an $\F$-basis $\{ E_i^*\mathbf{1}+W_{n+1}: i\in S_n\}$ of cardinality $|S_n|$.
\end{enumerate}
\end{lem}
\begin{proof}
By Definition \ref{D;primarymodule} and the definition of $S_n$, $W_m=\langle \{E_i^*\mathbf{1}: i\in \bigcup_{m\leq q\in \mathbb{N}}S_q\}\rangle_\F$ for every $n<m\in \mathbb{N}$. By hypotheses, $W_m=\{\mathbf{0}\}$ for every $n<m\in \mathbb{N}$. (i) is shown.
(ii) is proved by combining Lemma \ref{L;primemodule} (i), Definition \ref{D;primarymodule}, the definition of $S_n$.
\end{proof}
The following lemma contains more properties of the objects in Definition \ref{D;primarymodule}.
\begin{lem}\label{L;primarymoduleproperty}
The following statements hold:
\begin{enumerate}[(i)]
\item [\em (i)] $W_1$ is the unique maximal $T$-submodule of $W_0$.
\item [\em (ii)] $W_0$ is an indecomposable $T$-module.
\item [\em (iii)] $W_0/W_1$ is an irreducible $T$-module.
\item [\em (iv)] $W_0$ is an irreducible $T$-module if and only if $S$ is a $p'$-valenced scheme.
\item [\em (v)] Let $n\in \mathbb{N}_0$ and $U$ denote a $T$-submodule of $W_n/W_{n+1}$. Then there do not exist $T$-submodules $V$, $W$ of $W_{n+1}$ such that $W\subset V$ and $U\cong V/W$ as $T$-modules.
\end{enumerate}
\end{lem}
\begin{proof}
For (i), by the definitions of $W_0$ and $W_1$, $W_1$ is a proper $T$-submodule of $W_0$. Let $M$ denote a maximal $T$-submodule of $W_0$. According to the definition of $W_0$, we pick $\sum_{i=0}^dc_iE_i^*\mathbf{1}\in M$, where $c_i\in \F$ for every $i\in [0,d]$. For every $i\in [0,d]$, we claim that $c_i=0$ if $p\nmid k_i$. Suppose that there exists $j\in [0,d]$ such that $p\nmid k_j$ and $c_j\neq 0$. Let $\ell\in [0,d]$. Notice that $E_\ell^*JE_j^*\in T$ by \eqref{Eq;preliminary2} and the definition of $T$. Since $M$ is a $T$-submodule of $W_0$, notice that $E_\ell^*JE_j^*(\sum_{i=0}^dc_iE_i^*\mathbf{1})=c_jE_\ell^*JE_j^*\mathbf{1}=c_j\overline{k_j}E_\ell^*\mathbf{1}\in M$ by \eqref{Eq;preliminary3} and \eqref{Eq;preliminary5}. Hence $E_\ell^*\mathbf{1}\in M$ as $c_j\neq 0$ and $p\nmid k_j$. Since $\ell$ is chosen from $[0,d]$ arbitrarily and $M$ is a maximal $T$-submodule of $W_0$, we thus have $W_0\subseteq M\subset W_0$, which is a contradiction. So the desired claim follows. Since $\sum_{i=0}^dc_iE_i^*\mathbf{1}$ is chosen from $M$ arbitrarily, by the proven claim and the definition of $W_1$, observe that $M\subseteq W_1\subset W_0$, which implies that $M=W_1$ since $M$ is a maximal $T$-submodule of $W_0$. Since $M$ is chosen from the set of all maximal $T$-submodules of $W_0$ arbitrarily, $W_1$ is the unique maximal $T$-submodule of $W_0$. (i) is proved.

For (ii), $W_0\neq\{\mathbf{0}\}$ since $\dim_\F W_0=d+1>0$. Suppose that there exist nonzero $T$-modules $N_1$ and $N_2$ such that $W_0=N_1\oplus N_2$. Hence $W_0$ has at least two distinct maximal $T$-submodules, which contradicts (i). Therefore $W_0$ is an indecomposable $T$-module. (ii) is shown.

For (iii), note that (iii) is from (i).

For (iv), by (i) and (iii), $W_0$ is an irreducible $T$-module if and only if $W_1=\{\mathbf{0}\}$. By the definition of $W_1$, note that $W_1=\{\mathbf{0}\}$ if and only if $S$ is a $p'$-valenced scheme. (iv) thus follows.

For (v), there is no loss to assume further that $U\neq \{\mathbf{0}+W_{n+1}\}$. By Lemma \ref{L;submoduleseries} (ii), pick $\sum_{u\in S_n}e_uE_u^*\mathbf{1}+W_{n+1}\in U\setminus\{\mathbf{0}+W_{n+1}\}$, where $e_u\in \F$ for every $u\in S_n$. So there is $v\in S_n$ such that $e_v\neq 0$. As $U$ is a $T$-submodule of $W_n/W_{n+1}$, by \eqref{Eq;preliminary3} and Lemma \ref{L;submoduleseries} (ii), $\mathbf{0}+W_{n+1}\neq e_vE_v^*\mathbf{1}+W_{n+1}=E_v^*(\sum_{u\in S_n}e_uE_u^*\mathbf{1}+W_{n+1})\in U$. Suppose that there exist $T$-submodules $V$, $W$ of $W_{n+1}$ such that $W\subset V$ and $U\cong V/W$ as $T$-modules. Let $\phi$ be a $T$-isomorphism from $U$ and $V/W$. Since $\phi$ is injective and $\mathbf{0}+W_{n+1}\neq e_vE_v^*\mathbf{1}+W_{n+1}\in U$, notice that $\mathbf{0}+W\neq\phi(e_vE_v^*\mathbf{1}+W_{n+1})\in V/W$. Since $v\in S_n$, by the definitions of $S_n$ and $W_{n+1}$, \eqref{Eq;preliminary3} tells us that $E_v^*\in \mathrm{Ann}_T(W_{n+1})$. In particular, as $W\subset V\subseteq W_{n+1}$, notice that $E_v^*\phi(e_vE_v^*\mathbf{1}+W_{n+1})=\mathbf{0}+W$. As $\phi$ is a $T$-isomorphism, by \eqref{Eq;preliminary3} again, we thus can deduce that
$$\mathbf{0}+W\neq\phi(e_vE_v^*\mathbf{1}+W_{n+1})=\phi(E_v^*(e_vE_v^*\mathbf{1}+W_{n+1}))
=E_v^*\phi(e_vE_v^*\mathbf{1}+W_{n+1})=\mathbf{0}+W,$$
which is impossible. So there are not $T$-submodules $V$, $W$ of $W_{n+1}$ such that $W\subset V$ and $U\cong V/W$ as $T$-modules. (v) is proved.
\end{proof}
For further discussion, we need the following notation and lemma.
\begin{nota}\label{N;notation1}
Let $\sim$ be a binary relation on $[0,d]$, where, for any $i, j\in [0,d]$, $i\sim j$ if and only if the following conditions hold:
\begin{enumerate}[(i)]
\item There are $m\in \mathbb{N}_0$ and sequence $i_0, j_0, \ell_0, i_1, j_1, \ell_1,\ldots, i_m, j_m, \ell_m$ of $[0,d]$ such that $i_0=i$, $\ell_m=j$, $p\nmid \prod_{a=0}^mp_{\ell_aj_a'}^{i_a}$, and $\ell_{b-1}=i_b$ for every $b\in [1,m]$;
\item There are $n\in \mathbb{N}_0$ and sequence $r_0, s_0, t_0, r_1, s_1, t_1,\ldots, r_n, s_n, t_n$ of  $[0,d]$ such that $r_0=j$, $t_n=i$, $p\nmid \prod_{c=0}^np_{t_cs_c'}^{r_c}$, and $t_{e-1}=r_e$ for every $e\in [1,n]$.
\end{enumerate}
\end{nota}
\begin{eg}\label{E;notation1}
Let us illustrate the definition of $\sim$ by examples. Assume that $p>2$ and $S$ is the scheme of order $12$, No. $21$ in \cite{HM}. Notice that $S=\{R_0, R_1, R_2, R_3, R_4\}$, where $3'=p_{44}^3=p_{33}^4=4$.
So $3\sim 4$ as the sequence $3,3,4$ satisfies Notation \ref{N;notation1} (i) and the sequence $4,4,3$ satisfies Notation \ref{N;notation1} (ii). Notice that $3\sim 3$ as the sequence $3,3,4,4,4,3$ satisfies Notation \ref{N;notation1} (i) and (ii).
\end{eg}
\begin{lem}\label{L;equivalencerelation}
The binary relation $\sim$ is an equivalence relation on $[0,d]$.
\end{lem}
\begin{proof}
Let $i,j,\ell\in [0,d]$. Notice that $0'=0$ and $p_{i0}^i=1$ by the definition of $p_{i0}^i$. So the sequence $i, 0, i$ satisfies Notation \ref{N;notation1} (i) and (ii), which implies that $i\sim i$. Since $i$ is chosen from $[0,d]$ arbitrarily, $\sim$ is reflexive.

Assume that $i\sim j$. According to Notation \ref{N;notation1} (i) and (ii), we have the following two facts:
\begin{enumerate}[(i)]
\item [(i)]There are $m_1\in \mathbb{N}_0$ and sequence $i_0, j_0, \ell_0, i_1, j_1, \ell_1, \ldots, i_{m_1}, j_{m_1}, \ell_{m_1}$ of $[0,d]$ such that $i_0=i$, $\ell_{m_1}=j$, $p\nmid \prod_{a=0}^{m_1}p_{\ell_aj_a'}^{i_a}$, and $\ell_{b-1}=i_b$ for every $b\in [1, m_1]$.
\item [(ii)]There are $n_1\in \mathbb{N}_0$ and sequence $r_0, s_0, t_0, r_1, s_1, t_1,\ldots, r_{n_1}, s_{n_1}, t_{n_1}$ of $[0,d]$ such that $r_0=j$, $t_{n_1}=i$, $p\nmid \prod_{c=0}^{n_1}p_{t_cs_c'}^{r_c}$, and $t_{e-1}=r_e$ for every $e\in [1, n_1]$.
\end{enumerate}
We thus have $j\sim i$ as Notation \ref{N;notation1} (i) follows from (ii) and Notation \ref{N;notation1} (ii) follows from (i). So $\sim$ is symmetric.

Assume further that $j\sim \ell$. According to Notation \ref{N;notation1} (i) and (ii), we have the following two facts:
\begin{enumerate}[(i)]
\item [(iii)] There are $m_2\in \mathbb{N}_0$ and sequence $u_0, v_0, w_0, u_1, v_1, w_1,\ldots, u_{m_2}, v_{m_2}, w_{m_2}$ of $[0,d]$ such that $u_0=j$, $w_{m_2}=\ell$, $p\nmid \prod_{f=0}^{m_2}p_{w_fv_f'}^{u_f}$, and $w_{g-1}=u_g$ for every $g\in [1, m_2]$.
\item [(iv)] There are $n_2\in \mathbb{N}_0$ and sequence $x_0, y_0, z_0, x_1, y_1, z_1,\ldots, x_{n_2}, y_{n_2}, z_{n_2}$ of $[0,d]$ such that $x_0=\ell$, $z_{n_2}=j$, $p\nmid \prod_{h=0}^{n_2}p_{z_hy_h'}^{x_h}$, and $z_{k-1}=x_k$ for every $k\in [1, n_2]$.
\end{enumerate}
Set $i_{m_1+f+1}=u_f$, $j_{m_1+f+1}=v_f$, and $\ell_{m_1+f+1}=w_f$ for every $f\in [0, m_2]$. We also put $x_{n_2+c+1}=r_c$, $y_{n_2+c+1}=s_c$, and $z_{n_2+c+1}=t_c$ for every $c\in [0, n_1]$. By (i) and (iii), observe that $i_0=i$, $\ell_{m_1+m_2+1}=\ell$, $p\nmid \prod_{r=0}^{m_1+m_2+1}p_{\ell_rj_r'}^{i_r}$, and $\ell_{s-1}=i_s$ for every $s\in [1, m_1+m_2+1]$. According to (ii) and (iv), observe that $x_0=\ell$, $z_{n_1+n_2+1}=i$, $p\nmid \prod_{u=0}^{n_1+n_2+1}p_{z_uy_u'}^{x_u}$, and $z_{v-1}=x_v$ for every $v\in[1, n_1+n_2+1]$. We thus deduce that $i\sim \ell$ since the sequence $i_0, j_0,\ell_0, i_1, j_1, \ell_1,\ldots, i_{m_1+m_2+1}, j_{m_1+m_2+1}, \ell_{m_1+m_2+1}$ satisfies Notation \ref{N;notation1} (i) and the sequence $x_0, y_0, z_0, x_1, y_1, z_1, \ldots, x_{n_1+n_2+1}, y_{n_1+n_2+1}, z_{n_1+n_2+1}$ satisfies Notation \ref{N;notation1} (ii). Therefore $\sim$ is transitive. The desired lemma thus follows from the definition of an equivalence relation on a set.
\end{proof}
For further discussion, we also need the following notation and lemma.
\begin{nota}\label{N;notation2}
Let $n\in \mathbb{N}_0$ and $\sim_n$ denote a binary relation on $S_n$, where, for any $i,j \in S_n$, $i\sim_n j$ if and only if $i\sim j$. As $S_n\subseteq [0,d]$, by Lemma \ref{L;equivalencerelation}, notice that $\sim_n$ is an equivalence relation on $S_n$. If $S_n\neq \varnothing$, let $Q_n$ be the quotient set of $S_n$ with respect to $\sim_n$. If $S_n=\varnothing$, set $Q_n=\varnothing$.
\end{nota}
\begin{lem}\label{L;primarymoduleadditionalproperty}
Assume that $n\in \mathbb{N}_0$, $Q_n\neq \varnothing$, and $C\in Q_n$.
\begin{enumerate}[(i)]
\item [\em (i)] $\varnothing\neq C\subseteq S_n$.
\item [\em (ii)]$\langle\{E_i^*\mathbf{1}+W_{n+1}: i\in C\}\rangle_\F$ is a nonzero $T$-submodule of $W_n/W_{n+1}$.
\item [\em (iii)] $\langle\{E_i^*\mathbf{1}+W_{n+1}: i\in C\}\rangle_\F$ is an irreducible  $T$-submodule of $W_n/W_{n+1}$.
\end{enumerate}
\end{lem}
\begin{proof}
As $Q_n\neq \varnothing$, by the definition of $Q_n$, $Q_n$ is a partition of $S_n$. (i) thus follows.

By (i) and Lemma \ref{L;submoduleseries} (ii), $\{\mathbf{0}+W_{n+1}\}\neq \langle\{E_i^*\mathbf{1}+W_{n+1}: i\in C\}\rangle_\F\subseteq W_n/W_{n+1}$.
Let $a,b,c\in [0,d]$. We claim that $E_a^*A_bE_c^*(E_h^*\mathbf{1}+W_{n+1})\in \langle\{E_i^*\mathbf{1}+W_{n+1}: i\in C\}\rangle_\F$ for every $h\in C$. We suppose that $E_a^*A_bE_c^*(E_j^*\mathbf{1}+W_{n+1})\notin\langle\{E_i^*\mathbf{1}+W_{n+1}: i\in C\}\rangle_\F$ and $j\in C$. Therefore $E_a^*A_bE_c^*(E_j^*\mathbf{1}+W_{n+1})\neq \mathbf{0}+W_{n+1}$. So $E_a^*A_bE_c^*E_j^*\mathbf{1}\notin W_{n+1}$, which implies that $c=j$ and
\begin{align}\label{Eq;preliminary7}
\mathbf{0}\neq \overline{p_{jb'}^a}E_a^*\mathbf{1}=E_a^*A_bE_c^*E_j^*\mathbf{1}\notin W_{n+1}
\end{align}
by \eqref{Eq;preliminary3} and Lemma \ref{L;generalresults} (i). So we have $p\nmid p_{jb'}^a$. Notice that  $E_a^*\mathbf{1}+W_{n+1}\in W_n/W_{n+1}$ as $E_a^*A_bE_c^*(E_j^*\mathbf{1}+W_{n+1})\in W_n/W_{n+1}$, \eqref{Eq;preliminary7} holds, and $p\nmid p_{jb'}^a$. Moreover, notice that $E_a^*\mathbf{1}\notin W_{n+1}$ as \eqref{Eq;preliminary7} holds. Therefore $a\in S_n$ by Lemmas \ref{L;submoduleseries} (ii) and \ref{L;primemodule} (i). Since $j\in C$, by (i), notice that $j\in S_n$. As $a,j\in S_n$, $p\nmid p_{jb'}^a$, and $k_ap_{jb'}^a=k_jp_{ab}^j$ by Lemma \ref{L;trigular}, notice that $p\nmid p_{ab}^j$ by the definition of $S_n$. We thus have $j\sim a$ and $j\sim_n a$ since the sequence $j, b', a$ satisfies Notation \ref{N;notation1} (i) and the sequence $a,b,j$ satisfies Notation \ref{N;notation1} (ii). As $j\in C$, $j\sim_n a$, and $C$ is an equivalence class of $S_n$ with respect to $\sim_n$, notice that $a\in C$. So $E_a^*A_bE_c^*(E_j^*\mathbf{1}+W_{n+1})\in \langle\{E_i^*\mathbf{1}+W_{n+1}: i\in C\}\rangle_\F$ by \eqref{Eq;preliminary7}. We thus have a contradiction as $E_a^*A_bE_c^*(E_j^*\mathbf{1}+W_{n+1})\notin\langle\{E_i^*\mathbf{1}+W_{n+1}: i\in
C\}\rangle_\F$. The desired claim thus follows. As we have known that $\langle\{E_i^*\mathbf{1}+W_{n+1}: i\in C\}\rangle_\F$ is a nonzero $\F$-linear subspace of $W_n/W_{n+1}$ and $a,b,c$ are chosen from $[0,d]$ arbitrarily, (ii) is proved by combining Lemma \ref{L;generalresults} (iii), \eqref{Eq;preliminary6}, and the proven claim.

For (iii), we first suppose that the $T$-module $\langle\{E_i^*\mathbf{1}+W_{n+1}: i\in C\}\rangle_\F$ in (ii) has a nonzero proper $T$-submodule $U$. Pick $\sum_{i\in C}c_iE_i^*\mathbf{1}+W_{n+1}\in U\setminus\{\mathbf{0}\}$, where $c_i\in \F$ for every $i\in C$. So there exists $k\in C$ such that $c_k\neq 0$. Let $\ell \in C$. Notice that $\ell\sim_n k$ and $\ell\sim k$ as $C$ is an equivalence class of $S_n$ with respect to $\sim_n$. By Notation \ref{N;notation1} (i), there are $m\in \mathbb{N}_0$ and sequence $i_0, j_0, \ell_0, i_1, j_1, \ell_1,\ldots, i_m, j_m, \ell_m$ of $[0,d]$ such that $i_0=\ell$, $\ell_m=k$, $p\nmid \prod_{e=0}^mp_{\ell_ej_e'}^{i_e}$, and $\ell_{f-1}=i_f$ for every $f\in [1,m]$. Let $\gamma$ denote $\prod_{e=0}^m{p_{\ell_ej_e'}^{i_e}}$. Since $c_k\neq 0$ and $p\nmid \gamma$, notice that $c_k\overline{\gamma}\neq 0$. Since $U$ is a $T$-submodule of the $T$-module $\langle\{E_i^*\mathbf{1}+W_{n+1}: i\in C\}\rangle_\F$ in (ii), by \eqref{Eq;preliminary6}, \eqref{Eq;preliminary3}, and Lemma \ref{L;generalresults} (i), notice that
$\prod_{e=0}^m(E_{i_e}^*A_{j_e}E_{\ell_e}^*)(\sum_{i\in C}c_iE_i^*\mathbf{1}+W_{n+1})=
c_k\overline{\gamma}E_\ell^*\mathbf{1}+W_{n+1}\in U$. We thus have $E_\ell^*\mathbf{1}+W_{n+1}\in U$ as $c_k\overline{\gamma}\neq 0$. Since $\ell$ is chosen from $C$ arbitrarily, notice that $\langle\{E_i^*\mathbf{1}+W_{n+1}: i\in C\}\rangle_\F\subseteq U\subset\langle\{E_i^*\mathbf{1}+W_{n+1}: i\in C\}\rangle_\F$, which is impossible. So the $T$-module $\langle\{E_i^*\mathbf{1}+W_{n+1}: i\in C\}\rangle_\F$ in (ii) has no nonzero proper $T$-submodule. (iii) thus follows.
\end{proof}
We now can introduce the following notation.
\begin{nota}\label{N;notation3}
Assume that $n\in \mathbb{N}_0$, $Q_n\neq \varnothing$, and $C\in Q_n$. Let $Irr_n(C)$ denote the irreducible $T$-submodule of $W_n/W_{n+1}$ in Lemma \ref{L;primarymoduleadditionalproperty} (iii). By combining Lemmas \ref{L;primarymoduleadditionalproperty} (i), (iii), and \ref{L;submoduleseries} (ii), notice that $Irr_n(C)$ has an $\F$-basis $\{E_i^*\mathbf{1}+W_{n+1}: i\in C\}$ of cardinality $|C|$. Write $B_n(C)$ for $\{E_i^*\mathbf{1}+W_{n+1}: i\in C\}$.
\end{nota}
We need the following three lemmas to deduce the main result of this section.
\begin{lem}\label{L;nonisomorphism}
Assume that $n_m\in \mathbb{N}_0$, $Q_{n_m}\neq \varnothing$, and $C_m\in Q_{n_m}$ for every $m\in [1,2]$. Then $Irr_{n_1}(C_1)\cong Irr_{n_2}(C_2)$ as $T$-modules if and only if $n_1=n_2$ and $C_1=C_2$.
\end{lem}
\begin{proof}
If $n_1=n_2$ and $C_1=C_2$, $Irr_{n_1}(C_1)=Irr_{n_2}(C_2)$ by the definitions of $Irr_{n_1}(C_1)$ and $Irr_{n_2}(C_2)$. So $Irr_{n_1}(C_1)\cong Irr_{n_2}(C_2)$ as $T$-modules. Conversely, we
assume that $Irr_{n_1}(C_1)\cong Irr_{n_2}(C_2)$ as $T$-modules. By the definitions of $Irr_{n_1}(C_1)$ and $Irr_{n_2}(C_2)$, $Irr_{n_m}(C_m)$ is an irreducible $T$-submodule of $W_{n_m}/W_{n_m+1}$ for every $m\in[1,2]$.

If $n_1<n_2$, by the Correspondence Theorem for Modules, there is a $T$-submodule $U$ of $W_{n_2}$ such that $W_{n_2+1}\subset U\subseteq W_{n_1+1}$ and $U/W_{n_2+1}=Irr_{n_2}(C_2)\cong Irr_{n_1}(C_1)$ as $T$-modules. We thus have a contradiction by Lemma \ref{L;primarymoduleproperty} (v). If $n_1>n_2$, according to the Correspondence Theorem for Modules again, there exists a $T$-submodule $V$ of $W_{n_1}$ such that $W_{n_1+1}\subset V\subseteq W_{n_2+1}$ and $V/W_{n_1+1}=Irr_{n_1}(C_1)\cong Irr_{n_2}(C_2)$ as $T$-modules. We also have a contradiction by Lemma \ref{L;primarymoduleproperty} (v). So $n_1=n_2$.

Set $n=n_1=n_2$. Let $\phi$ denote a $T$-isomorphism from $Irr_n(C_1)$ to $Irr_n(C_2)$. Pick $i\in C_1$. Since $B_n(C_1)$ is an $\F$-basis of $Irr_n(C_1)$, $\mathbf{0}+W_{n+1}\neq E_i^*\mathbf{1}+W_{n+1}\in B_n(C_1)$. As $\phi$ is injective, we thus have $\mathbf{0}+W_{n+1}\neq \phi(E_i^*\mathbf{1}+W_{n+1})\in Irr_n(C_2)$. Moreover, as $B_n(C_2)$ is an $\F$-basis of $Irr_n(C_2)$, $\phi(E_i^*\mathbf{1}+W_{n+1})$ is an $\F$-linear combination of the elements of $B_n(C_2)$. Suppose that $C_1\neq C_2$. Since $C_1$ and $C_2$ are distinct equivalence classes of $S_n$ with respect to $\sim_n$, notice that $C_1\cap C_2=\varnothing$ and $i\notin C_2$. By \eqref{Eq;preliminary3}, we thus have $E_i^*(E_j^*\mathbf{1}+W_{n+1})=\mathbf{0}+W_{n+1}$ for every $E_j^*\mathbf{1}+W_{n+1}\in B_n(C_2)$. As \eqref{Eq;preliminary3} holds and $\phi$ is a $T$-isomorphism, we thus can deduce that
$$\mathbf{0}+W_{n+1}\neq \phi(E_i^*\mathbf{1}+W_{n+1})= \phi(E_i^*(E_i^*\mathbf{1}+W_{n+1}))=E_i^*\phi(E_i^*\mathbf{1}+W_{n+1})=\mathbf{0}+W_{n+1},$$
which is a contradiction. We thus have $C_1=C_2$. The proof is now complete.
\end{proof}
\begin{lem}\label{L;decomposition}
Let $n\in \mathbb{N}_0$.
\begin{enumerate}[(i)]
\item [\em (i)] $W_n/W_{n+1}=\{\mathbf{0}+W_{n+1}\}$ if and only if $Q_n=\varnothing$.
\item [\em (ii)] If $Q_n\neq \varnothing$, then $W_{n}/W_{n+1}=\bigoplus_{C\in Q_n}Irr_n(C)$.
\item [\em (iii)] If $Q_n$ contains precisely $m$ elements $C_1, C_2,\ldots, C_m$, then there is a $T$-submodule series $W_{n+1}=U_m\subset U_{m-1}\subset\cdots\subset U_1\subset U_0=W_n$ of $W_n$ such that $U_{q-1}/U_q\cong Irr_n(C_q)$ as $T$-modules for every $q\in [1,m]$.
\end{enumerate}
\end{lem}
\begin{proof}
For (i), by the definition of $Q_n$, observe that $Q_n=\varnothing$ if and only if $S_n=\varnothing$. We also have $\dim_\F W_n/W_{n+1}=|S_n|$ by Lemma \ref{L;submoduleseries} (ii). So $W_n/W_{n+1}=\{\mathbf{0}+W_{n+1}\}$ if and only if $Q_n=\varnothing$. The proof of (i) is now complete.

For (ii), as $Q_n\neq \varnothing$, the definition of $Q_n$ tells us that $Q_n$ is a partition of $S_n$. By Lemmas \ref{L;primarymoduleadditionalproperty} (i) and \ref{L;submoduleseries} (ii), we thus have $\bigcup_{C\in Q_n}B_n(C)=\{E_i^*\mathbf{1}+W_{n+1}: i\in S_n\}$, where $B_n(C^{(1)})\cap B_n(C^{(2)})=\varnothing$ if $C^{(1)}, C^{(2)}\in Q_n$ and  $C^{(1)}\neq C^{(2)}$. By Lemma \ref{L;submoduleseries} (ii) again, we thus can deduce that $W_n/W_{n+1}=\bigoplus_{C\in Q_n}\langle B_n(C)\rangle_\F=\bigoplus_{C\in Q_n}Irr_n(C)$. The proof of (ii) is now complete.

For (iii), we set $V_q=Irr_n(C_q)\oplus Irr_n(C_{q+1})\oplus\cdots\oplus Irr_n(C_m)$ for every $q\in [1,m]$. By (ii), observe that $W_n/W_{n+1}=V_1$. By the Correspondence Theorem for Modules, there exists a $T$-submodule series $W_{n+1}=U_m\subset U_{m-1}\subset \cdots \subset U_1\subset U_0=W_n$ of $W_n$ such that $U_{q-1}/U_m=V_q$ for every $q\in [1,m]$. In particular, $U_{m-1}/U_m\cong Irr_n(C_m)$ as $T$-modules. For every $r\in [1, m-1]$, by the Third Isomorphism Theorem, note that $U_{r-1}/U_r\cong (U_{r-1}/U_m)/(U_r/U_m)=V_r/V_{r+1}\cong Irr_n(C_r)$ as $T$-modules. The proof of (iii) is now complete.
\end{proof}
\begin{rem}\label{R;remark1}
We have $|Q_0|=1$ and $Q_0=\{S_0\}$ by Lemmas \ref{L;primarymoduleproperty} (iii) and \ref{L;decomposition} (ii).
\end{rem}
\begin{lem}\label{L;filtration}
Let $\epsilon=\max\{m\in \mathbb{N}_0: \exists\ i\in [0,d],\ p^m\mid k_i\}$. Then there exists a $T$-submodule series $\{\mathbf{0}\}=W_{\epsilon+1}\subset W_\epsilon\subseteq W_{\epsilon-1}\subseteq\cdots\subseteq W_1\subset W_0$ of $W_0$, where, for every $n\in [0,\epsilon]$,
\begin{align*}
W_n/W_{n+1}=\begin{cases}\bigoplus_{C\in Q_n}Irr_n(C), & \text{if}\ Q_n\neq \varnothing,\\
\{\mathbf{0}+W_{n+1}\}, & \text{if}\ Q_n=\varnothing.
\end{cases}
\end{align*}
\end{lem}
\begin{proof}
By the definitions of $\epsilon$ and $S_\epsilon$, observe that $S_\epsilon\neq \varnothing=S_q$ for every $\epsilon<q\in \mathbb{N}$. So $W_{\epsilon+1}=\{\mathbf{0}\}\subset W_\epsilon$ by Lemma \ref{L;submoduleseries} (i) and (ii). So the desired $T$-module series follows from Definition \ref{D;primarymodule}. The desired equality is from Lemma \ref{L;decomposition} (i) and (ii).
\end{proof}
We are now ready to deduce the main result of this section.
\begin{thm}\label{T;primarymodule}
Let $\epsilon=\max\{m\in \mathbb{N}_0: \exists\ i\in [0,d],\ p^m\mid k_i\}$. Let $Q$ denote the set $\{n\in [0,\epsilon]: Q_n\neq \varnothing\}$.
\begin{enumerate}[(i)]
\item [\em (i)] $\bigcup_{q\in Q}\bigcup_{C\in Q_q}\{Irr_q(C)\}$ is the set of all composition factors of $W_0$ with respect to a composition series of $W_0$. Furthermore, $\bigcup_{q\in Q}\bigcup_{C\in Q_q}\{Irr_q(C)\}$ is a complete set of representatives of isomorphic classes of all composition factors of $W_0$.
\item[\em (ii)] $\sum_{q\in Q}|Q_q|$ equals the composition length of $W_0$. Furthermore, $\sum_{q\in Q}|Q_q|$ equals the number of isomorphic classes of all composition factors of $W_0$.
\end{enumerate}
\end{thm}
\begin{proof}
For (i), by Lemmas \ref{L;filtration} and \ref{L;decomposition} (iii), there exists a composition series of $W_0$ whose successive subquotients are precisely the elements of $\bigcup_{q\in Q}\bigcup_{C\in Q_q}\{Irr_q(C)\}$. So the first statement of (i) holds. By Lemma \ref{L;nonisomorphism}, any two distinct elements of $\bigcup_{q\in Q}\bigcup_{C\in Q_q}\{Irr_q(C)\}$ are not isomorphic to each other as $T$-modules. So the second statement of (i) is from the first one and the Jordan-H\"{o}lder Theorem. (i) is shown.

For (ii), by the first statement of (i) and the definition of the composition length of $W_0$,  the composition length of $W_0$ is $|\bigcup_{q\in Q}\bigcup_{C\in Q_q}\{Irr_q(C)\}|$. By Lemma \ref{L;nonisomorphism},
\begin{align}\label{Eq;preliminary8}
|\bigcup_{q\in Q}\bigcup_{C\in Q_q}\{Irr_q(C)\}|=\sum_{q\in Q}\sum_{C\in Q_q}1=\sum_{q\in Q}|Q_q|.
\end{align}
The first statement of (ii) thus follows. The second statement of (ii) comes from the second statement of (i) and \eqref{Eq;preliminary8}. (ii) is proved.
\end{proof}
\begin{rem}\label{R;remark2}
Let $A$ be a finite-dimensional associative unital $\F$-algebra. Call an $A$-module a multiplicity free $A$-module if its composition length equals the number of isomorphic classes of all its composition factors. By Lemma \ref{L;primarymoduleproperty} (ii) and Theorem \ref{T;primarymodule} (ii), notice that $W_0$ is an indecomposable multiplicity free $T$-module.
\end{rem}
\begin{eg}\label{E;theorem1}
Let us illustrate Theorem \ref{T;primarymodule} by an example. Assume that $p=2$ and $S$ is the scheme in Example \ref{E;notation1}. Notice that $k_0=k_1=1$, $k_2=2$, $k_3=k_4=4$, $S_0=\{0,1\}$, $S_1=\{2\}$, $S_2=\{3,4\}$, $\epsilon=2$, and $Q=\{0,1,2\}$. By Remark \ref{R;remark1}, note that $Q_0=\{\{0,1\}\}$ and $Q_1=\{\{2\}\}$. Let $i,j\in [0,4]$. For every $\ell\in [2,4]$, $2\nmid p_{\ell j}^i$ if and only if $i=\ell$ and $j\in [0,1]$, which implies that $Q_2=\{\{3\},\{4\}\}$ and $\langle \{E_\ell^*\mathbf{1}\}\rangle_\F$ is a $T$-submodule of $W_0$ for every $\ell\in [2,4]$. According to Theorem \ref{T;primarymodule} (i), notice that $\{Irr_0(\{0,1\}), Irr_1(\{2\}), Irr_2(\{3\}), Irr_2(\{4\})\}$ is a complete set of representatives of isomorphic classes of all composition factors of $W_0$. By Theorem \ref{T;primarymodule} (ii), note that the composition length of $W_0$ is four.
\end{eg}
The following corollary is an application of Lemmas \ref{L;decomposition} and \ref{L;filtration}.
\begin{cor}\label{C;uniserial}
The following statements are equivalent:
\begin{enumerate}[(i)]
\item [\em (i)] For any $T$-submodules $U$ and $V$ of $W_0$, we have either $U\subseteq V$ or $V\subseteq U$;
\item [\em (ii)] For every $n\in \mathbb{N}_0$, we have either $W_{n+1}=W_n$ or $W_{n+1}$ is the unique maximal $T$-submodule of $W_n$.
\end{enumerate}
\end{cor}
\begin{proof}
We prove (ii) by (i). Let $q\in \mathbb{N}_0$. Assume further that $W_q\neq W_{q+1}$. Therefore $W_{q+1}\subset W_q$ by Definition \ref{D;primarymodule}. Suppose that $W_{q+1}$ is not a maximal $T$-submodule of $W_q$. The nonzero $T$-module $W_q/W_{q+1}$ is not an irreducible $T$-module. By combining Lemma \ref{L;decomposition} (i), (ii), and the Correspondence Theorem for Modules, notice that $W_q$ has at least two distinct maximal $T$-submodules. Let $M$ and $N$ be distinct maximal $T$-submodules of $W_q$. By (i), notice that either $M\subseteq N\subset W_0$ or $N\subseteq M\subset W_0$. We thus have a contradiction since $M$ and $N$ are distinct maximal $T$-submodules of $W_q$. So $W_{q+1}$ is a maximal $T$-submodule of $W_q$. By (i), we can also get a similar contradiction if we suppose that the maximal $T$-submodule $W_{q+1}$ is not the unique maximal $T$-submodule of $W_q$. As $q$ is chosen from $\mathbb{N}_0$ arbitrarily, (ii) thus follows.

We prove (i) by (ii). Let $\epsilon=\max\{m\in \mathbb{N}_0: \exists\ i\in [0,d],\ p^m\mid k_i\}$. Let $W$ denote a nonzero $T$-submodule of $W_0$. We have $\{\mathbf{0}\}=W_{\epsilon+1}\subset W_\epsilon\subseteq W_{\epsilon-1}\subseteq\cdots\subseteq W_1\subset W_0$ by Lemma \ref{L;filtration}. So there exists $r\in [0, \epsilon]$ such that $W\subseteq W_r$ and $W\nsubseteq W_{r+1}$. Hence $W_r\neq W_{r+1}$, which implies that $W_{r+1}$ is the unique maximal $T$-submodule of $W_r$ by (ii). Suppose that $W\neq W_r$. Then $W\subseteq W_{r+1}$ since $W\subset W_r$ and $W_{r+1}$ is the unique maximal $T$-submodule of $W_r$. We thus have a contradiction since $W\nsubseteq W_{r+1}$. Therefore $W=W_r$. As $W$ is an arbitrarily chosen nonzero $T$-submodule of $W_0$, every $T$-submodule of $W_0$ is one of the $T$-modules $W_0, W_1, \ldots, W_{\epsilon+1}$. (i) thus follows from Lemma \ref{L;filtration}.
\end{proof}
We end this section with the following remarks.
\begin{rem}\label{R;remark3}
Let $A$ be a finite-dimensional associative unital $\F$-algebra. Call an $A$-module a uniserial $A$-module if, for any $A$-submodules $U$ and $V$ of this $A$-module, we have either $U\subseteq V$ or $V\subseteq U$. Corollary \ref{C;uniserial} describes a criterion to determine whether $W_0$ is a uniserial $T$-module.
\end{rem}
\begin{rem}\label{R;remark4}
Observe that $W_0$ is a uniserial $T$-module if $d=1$. In general, $W_0$ may not be a uniserial $T$-module. In Example \ref{E;theorem1}, observe that $W_1$ is a direct sum of its $T$-submodules $\langle\{E_2^*\mathbf{1}\}\rangle_\F$, $\langle\{E_3^*\mathbf{1}\}\rangle_\F$, $\langle\{E_4^*\mathbf{1}\}\rangle_\F$. Hence $W_0$ of this case is not a uniserial $T$-module by Lemma \ref{L;primemodule} (i) and the definition of a uniserial $T$-module.
\end{rem}
\section{Some characterizations of $p'$-valenced schemes}
In this section, we characterize the $p'$-valenced schemes by some properties of their modular Terwilliger algebras. We recall the notations in Definition \ref{D;primarymodule} and Notation \ref{N;notation3}. By \eqref{Eq;preliminary2}, notice that $E_i^*JE_j^*\in T$ for any $i,j\in [0,d]$. We first present a lemma.
\begin{lem}\label{L;primaryideal}
The following statements hold:
\begin{enumerate}[(i)]
\item [\em (i)] $\{E_i^*JE_j^*: i,j\in [0,d]\}$ is an $\F$-linearly independent set of cardinality $(d+1)^2$.
\item [\em (ii)] $\langle\{E_i^*JE_j^*: i,j\in [0,d]\}\rangle_\F$ is a two-sided ideal of $T$.
\item [\em (iii)] $\langle\{E_i^*JE_j^*: i,j\in [0,d],\ p\mid k_ik_j\}\rangle_\F$ is a two-sided ideal of $T$.
\item [\em (iv)] Assume that $\ell\in [0,d]$. Then $\langle\{E_i^*JE_\ell^*: i\in [0,d]\}\rangle_\F$ is a $T$-module under the left multiplication action of $T$.
\end{enumerate}
\end{lem}
\begin{proof}
For (i), suppose that $\sum_{i=0}^d\sum_{j=0}^dc_{ij}E_i^*JE_j^*=O$, where $\bigcup_{i=0}^d\bigcup_{j=0}^d\{c_{ij}\}\subseteq\F$ and $(\bigcup_{i=0}^d\bigcup_{j=0}^d\{c_{ij}\})\cap (\F\setminus\{0\})\neq \varnothing$. So there exist $g,h \in [0,d]$ such that $c_{gh}\neq 0$. Hence $c_{gh}E_g^*JE_h^*=E_g^*(\sum_{i=0}^d\sum_{j=0}^dc_{ij}E_i^*JE_j^*)E_h^*=E_g^*OE_h^*=O$ by
\eqref{Eq;preliminary3}. By \eqref{Eq;preliminary4}, we thus deduce that $c_{gh}=0$, which contradicts the inequality $c_{gh}\neq 0$. Therefore $\{E_i^*JE_j^*: i,j\in [0,d]\}$ is an $\F$-linearly independent set. According to \eqref{Eq;preliminary3} and \eqref{Eq;preliminary4} again, we also note that $|\{E_i^*JE_j^*: i,j\in [0,d]\}|=(d+1)^2$. (i) thus follows.

Let $a,b,c\in [0,d]$. For (ii), according to \eqref{Eq;preliminary1}, \eqref{Eq;preliminary3}, and Lemma \ref{L;generalresults} (i), observe that $E_u^*JE_v^*E_a^*A_bE_c^*=(E_c^*A_{b'}E_a^*E_v^*JE_u^*)^t=(\delta_{av}\overline{p_{v b}^c}E_c^*JE_u^*)^t=\delta_{av}\overline{p_{vb}^c}E_u^*JE_c^*$ and $E_a^*A_bE_c^*E_u^*JE_v^*=\delta_{cu}\overline{p_{ub'}^a}E_a^*JE_v^*$ for any $u, v\in [0,d]$. In particular, $E_a^*A_bE_c^*E_u^*JE_v^*$ and $E_u^*JE_v^*E_a^*A_bE_c^*$ are contained in $\langle\{E_i^*JE_j^*: i,j\in [0,d]\}\rangle_\F$ for any $u,v \in [0,d]$. Since $\langle\{E_i^*JE_j^*: i,j\in [0,d]\}\rangle_\F$ is an $\F$-linear space and $a, b,c$ are chosen from $[0,d]$ arbitrarily, (ii) thus follows from Lemma \ref{L;generalresults} (iii) and \eqref{Eq;preliminary6}.

For (iii), notice that (iii) is trivial if $\langle\{E_i^*JE_j^*: i,j\in [0,d],\ p\mid k_ik_j\}\rangle_\F=\{O\}$. We thus assume further that $\langle\{E_i^*JE_j^*: i,j\in [0,d],\ p\mid k_ik_j\}\rangle_\F\neq\{O\}$. Let $r,s\in [0,d]$ and $p\mid k_r k_s$. Observe that $p\mid k_ck_r$ if $p\nmid p_{sb}^c$. Otherwise, suppose that $p\nmid p_{sb}^c$ and $p\nmid k_ck_r$. As $p\mid k_r k_s$ and $p\nmid k_ck_r$, notice that $p\nmid k_c$ and $p\mid k_s$. As $k_cp_{sb}^c=k_sp_{cb'}^s$ by Lemma \ref{L;trigular}, we thus have $p\mid p_{sb}^c$, which contradicts the assumption $p\nmid p_{sb}^c$. So we have $E_c^*A_{b'}E_a^*E_s^*JE_r^*=\delta_{as}\overline{p_{sb}^c}E_c^*JE_r^*\in \langle\{E_i^*JE_j^*: i,j\in [0,d],\ p\mid k_ik_j\}\rangle_\F$ by \eqref{Eq;preliminary3} and Lemma \ref{L;generalresults} (i). So $E_r^*JE_s^*E_a^*A_bE_c^*\in \langle\{E_i^*JE_j^*: i,j\in [0,d],\ p\mid k_ik_j\}\rangle_\F$ as $E_r^*JE_s^*E_a^*A_bE_c^*=(E_c^*A_{b'}E_a^*E_s^*JE_r^*)^t=\delta_{as}\overline{p_{sb}^c}E_r^*JE_c^*$ by \eqref{Eq;preliminary1}. Notice that $p\mid k_ak_s$ if $p\nmid p_{rb'}^a$. Otherwise, suppose that $p\nmid p_{rb'}^a$ and $p\nmid k_ak_s$. Since $p\mid k_rk_s$ and $p\nmid k_ak_s$, note that $p\nmid k_a$ and $p\mid k_r$. Since $k_ap_{rb'}^a=k_rp_{ab}^r$ by Lemma \ref{L;trigular}, we thus have $p\mid p_{rb'}^a$, which contradicts the assumption $p\nmid p_{rb'}^a$. By \eqref{Eq;preliminary3} and Lemma \ref{L;generalresults} (i), we thus have $E_a^*A_bE_c^*E_r^*JE_s^*=\delta_{cr}\overline{p_{rb'}^a}E_a^*JE_s^*\in \langle\{E_i^*JE_j^*: i,j\in [0,d],\ p\mid k_ik_j\}\rangle_\F$. We thus deduce that $E_a^*A_bE_c^*E_u^*JE_v^*, E_u^*JE_v^*E_a^*A_bE_c^*\in \langle\{E_i^*JE_j^*: i,j\in [0,d],\ p\mid k_ik_j\}\rangle_\F$ for any $u,v\in [0,d]$ and $p\mid k_uk_v$. Since $\langle\{E_i^*JE_j^*: i,j\in [0,d],\ p\mid k_ik_j\}\rangle_\F$ is an $\F$-linear space and $a,b,c$ are chosen from $[0,d]$ arbitrarily, (iii) thus follows from Lemma \ref{L;generalresults} (iii) and \eqref{Eq;preliminary6}.

For (iv), by \eqref{Eq;preliminary3} and Lemma \ref{L;generalresults} (i),  $E_a^*A_bE_c^*E_u^*JE_\ell^*=\delta_{cu}\overline{p_{ub'}^a}E_a^*JE_\ell^*$ for every $u\in [0,d]$. Since $\langle\{E_i^*JE_\ell^*: i\in [0,d]\}\rangle_\F$ is an $\F$-linear space and $a,b,c$ are chosen from $[0,d]$ arbitrarily, (iv) thus follows from Lemma \ref{L;generalresults} (iii) and \eqref{Eq;preliminary6}.
\end{proof}
The following notation will be heavily used in the following discussion.
\begin{nota}\label{N;notation4}
Denote the two-sided ideals of $T$ in Lemma \ref{L;primaryideal} (ii) and (iii) by $B_0$ and $B_1$, respectively. So $B_0$ is a $T$-module under the left multiplication action of $T$. Let $_{T}B_0$ denote this $T$-module. Notice that $T$ itself is also a $T$-module under the left multiplication action of $T$. Let $_{T}T$ denote this $T$-module. Let $\ell\in [0,d]$. Denote the $T$-module in Lemma \ref{L;primaryideal} (iv) by $M_\ell$. Observe that $B_1\subset B_0\subseteq T$, $M_\ell\subseteq{_{T}B_0}\subseteq {_{T}T}$, $\dim_\F B_0=\dim_\F{_{T}B_0}=(d+1)^2$, $\dim_\F M_{\ell}=d+1$, and $B_0$ is an $\F$-subalgebra of $T$.
\end{nota}
The following lemma summarizes some properties of the objects in Notation \ref{N;notation4}.
\begin{lem}\label{L;properties}
The following statements hold:
\begin{enumerate}[(i)]
\item [\em (i)] $B_1$ is the unique maximal two-sided ideal of the $\F$-subalgebra $B_0$ of $T$.
\item [\em (ii)] The matrix product of any three elements of $B_1$ is $O$.
\item [\em (iii)]$B_1$ is a nilpotent two-sided ideal of $T$. In particular, $B_1\subseteq \mathrm{Rad}(T)$.
\item [\em (iv)] Assume that $\ell\in [0,d]$. Then $M_\ell\cong W_0$ as $T$-modules.
\item [\em (v)] $_{T}B_0$ is isomorphic to a direct sum of $d+1$ copies of $W_0$ as $T$-modules.
\end{enumerate}
\end{lem}
\begin{proof}
For (i), by the definitions of $B_0$ and $B_1$, $B_1$ is a proper two-sided ideal of the $\F$-subalgebra $B_0$ of $T$. Let $M$ be a maximal two-sided ideal of the $\F$-subalgebra $B_0$ of $T$. According to the definition of $B_0$, pick $\sum_{i=0}^d\sum_{j=0}^dc_{ij}E_i^*JE_j^*\in M$, where $c_{ij}\in \F$ for any $i,j\in [0,d]$. For any $i,j\in [0,d]$, we claim that $c_{ij}=0$ if $p\nmid k_ik_j$. Suppose that there exist $a,b\in [0,d]$ such that $p\nmid k_ak_b$ and $c_{ab}\neq 0$. Let $c,e\in [0,d]$. Observe that $E_c^*JE_a^*, E_b^*JE_e^*\in B_0$ by the definition of $B_0$. Since $M$ is a two-sided ideal of the $\F$-subalgebra $B_0$ of $T$, $c_{ab}\overline{k_ak_b}E_c^*JE_e^*=E_c^*JE_a^*(\sum_{i=0}^d\sum_{j=0}^dc_{ij}E_i^*JE_j^*)E_b^*JE_e^*\in M$ by \eqref{Eq;preliminary3} and \eqref{Eq;preliminary5}. Hence $E_c^*JE_e^*\in M$ as $c_{ab}\neq 0$ and $p\nmid k_ak_b$. Since $c,e$ are chosen from $[0,d]$ arbitrarily and $M$ is a maximal two-sided ideal of the $\F$-subalgebra $B_0$ of $T$, we thus have $B_0\subseteq M\subset B_0$, which is impossible. Hence the desired claim follows. Since $\sum_{i=0}^d\sum_{j=0}^dc_{ij}E_i^*JE_j^*$ is chosen from $M$ arbitrarily, by the proven claim and the definition of $B_1$, notice that $M\subseteq B_1\subset B_0$, which implies that $M=B_1$ as $M$ is a maximal two-sided ideal of the $\F$-subalgebra $B_0$ of $T$. As $M$ is an arbitrarily chosen maximal two-sided ideal of the $\F$-subalgebra $B_0$ of $T$, $B_1$ is the unique maximal two-sided ideal of the $\F$-subalgebra $B_0$ of $T$. (i) is proved.

For (ii), notice that (ii) is trivial if $B_1=\{O\}$. We assume further that $B_1\neq\{O\}$. Let $g,h, r,s ,u, v\in [0,d]$, where $p\mid k_gk_h$, $p\mid k_rk_s$, and $p\mid k_uk_v$. We thus deduce that
\begin{align}\label{Eq;preliminary9}
E_g^*JE_h^*E_r^*JE_s^*E_u^*JE_v^*=\delta_{hr}\delta_{su}E_g^*JE_r^*JE_s^*JE_v^*=\delta_{hr}\delta_{su}\overline{k_rk_s}E_g^*JE_v^*=O
\end{align}
by \eqref{Eq;preliminary3} and \eqref{Eq;preliminary5}. As $g,h,r,s,u,v$ are chosen from $[0,d]$ arbitrarily, (ii) thus follows from \eqref{Eq;preliminary9} and the definition of $B_1$.

For (iii), as $B_1$ is a two-sided ideal of $T$, (iii) is shown by (ii).

For (iv), by the definitions of $W_0$ and $M_\ell$, let $\phi$ be the $\F$-linear homomorphism from $W_0$ to $M_\ell$ that sends every $E_i^*\mathbf{1}$ to $E_i^*JE_\ell^*$. By Lemmas \ref{L;primemodule} (i) and \ref{L;primaryideal} (i), $\phi$ is an $\F$-linear isomorphism. Observe that $\phi$ is also a $T$-isomorphism by combining the definition of $W_0$, the definition of $M_\ell$, Lemma \ref{L;generalresults} (i), (iii), and \eqref{Eq;preliminary6}. We are done.

For (v), by the definition of $_{T}B_0$ and Lemma \ref{L;primaryideal} (i), notice that $_{T}B_0=\bigoplus_{i=0}^dM_i$. (v) thus follows from (iv).
\end{proof}
\begin{rem}\label{R;remark5}
In general, the matrix product of any two elements of $B_1$ may not be $O$. Assume that $S$ is not a $p'$-valenced scheme. Then there exists $R_i\in S$ such that $p\mid k_i$. Notice that $O\neq E_i^*JE_i^* =(E_i^*JE_0^*+E_0^*JE_i^*)^2\in B_1$ by \eqref{Eq;preliminary3}, \eqref{Eq;preliminary4},  \eqref{Eq;preliminary5}.
\end{rem}
\begin{rem}\label{R;remark6}
The containment in Lemma \ref{L;properties} (iii) may be strict (see \cite[5.1]{Han1}). The containment in Lemma \ref{L;properties} (iii) may become equality (see \cite[Theorems B and C]{Jiang}).
\end{rem}
The following lemma contains some characterizations of the $p'$-valenced schemes.
\begin{lem}\label{L;characterization1}
The following statements are equivalent:
\begin{enumerate}[(i)]
\item [\em (i)] $S$ is a $p'$-valenced scheme;
\item [\em (ii)] The $\F$-subalgebra $B_0$ of $T$ is unital. Its identity element is central in $T$;
\item [\em (iii)] There exists a two-sided ideal $D$ of $T$ such that $T$ is a direct sum of $B_0$ and $D$;
\item [\em (iv)] The $\F$-subalgebra $B_0$ of $T$ is isomorphic to a full matrix algebra over a division $\F$-algebra as $\F$-algebras.
\end{enumerate}
\end{lem}
\begin{proof}
We prove (ii) by (i). Set $e_{B_0}=\sum_{i=0}^d\overline{k_i}^{-1}E_i^*JE_i^*\in B_0$ by (i) and the definition of $B_0$. By \eqref{Eq;preliminary3} and \eqref{Eq;preliminary5}, $E_a^*JE_b^*e_{B_0}=e_{B_0}E_a^*JE_b^*=E_a^*JE_b^*$ for any $a,b\in [0,d]$. So $e_{B_0}$ is the identity element of the $\F$-subalgebra $B_0$ of $T$ by the definition of $B_0$. By combining \eqref{Eq;preliminary3}, \eqref{Eq;preliminary1}, Lemmas \ref{L;generalresults} (i), and \ref{L;trigular}, notice that
$$E_a^*A_bE_c^*e_{B_0}=\overline{k_c}^{-1}\overline{p_{cb'}^a}E_a^*JE_c^*=\overline{k_a}^{-1}\overline{p_{ab}^c}E_a^*JE_c^*
=(E_c^*A_{b'}E_a^*e_{B_0})^t=e_{B_0}E_a^*A_bE_c^*$$
for any $a,b,c\in [0,d]$. So $e_{B_0}$ is a central element of $T$ by Lemma \ref{L;generalresults} (iii) and \eqref{Eq;preliminary6}. The proof of (ii) is now complete.

We prove (iii) by (ii). By (ii), there exists $f_{B_0}\in B_0$ such that $f_{B_0}$ is the identity element of the $\F$-subalgebra $B_0$ of $T$ and $f_{B_0}$ is a central element of $T$. Let $D$ denote $\{(I-f_{B_0})Z: Z\in T\}$. As $f_{B_0}$ is a central element of $T$, notice that $D$ is a two-sided ideal of $T$ by the definition of $D$. As $f_{B_0}$ is the identity element of the $\F$-subalgebra $B_0$ of $T$, notice that $f_{B_0}^2=f_{B_0}$, $(I-f_{B_0})^2=I-f_{B_0}$, and $f_{B_0}(I-f_{B_0})=O$, which implies that $B_0\cap D=\{O\}$ by the definition of $D$. As $B_0$ is a two-sided ideal of $T$ and $Z=f_{B_0}Z+(I-f_{B_0})Z$ for every $Z\in T$, we thus get that $T$ is a direct sum of $B_0$ and $D$. The proof of (iii) is now complete.

We prove (i) by (iii). Since $I\in T$ and $B_0$ is a two-sided ideal of $T$, by (iii), notice that the $\F$-subalgebra $B_0$ of $T$ is a unital $\F$-algebra. By the definition of $B_0$, assume that $\sum_{i=0}^d\sum_{j=0}^dc_{ij}E_i^*JE_j^*$ is the identity element of the $\F$-subalgebra $B_0$ of $T$, where $c_{ij}\in\F$ for any $i,j\in [0,d]$. Suppose that $S$ is not a $p'$-valenced scheme. Then there exists $\ell\in [0,d]$ such that $p\mid k_\ell$. Moreover, observe that $O\neq E_\ell^*JE_\ell^*\in B_0$ by \eqref{Eq;preliminary4}. Since $\sum_{i=0}^d\sum_{j=0}^dc_{ij}E_i^*JE_j^*$ is the identity element of the $\F$-subalgebra $B_0$ of $T$, by \eqref{Eq;preliminary3} and \eqref{Eq;preliminary5}, we thus get that $O\neq E_\ell^*JE_\ell^*=E_\ell^*JE_\ell^*(\sum_{i=0}^d\sum_{j=0}^dc_{ij}E_i^*JE_j^*)=O$, which is a contradiction. (i) thus follows.

We prove (iv) by (i). Notice that $B_1=\{O\}$ by (i) and the definition of $B_1$. As (i) implies (ii), by Lemma \ref{L;properties} (i), the $\F$-subalgebra $B_0$ of $T$ is a simple unital $\F$-algebra. So (iv) follows from the Artin-Wedderburn Theorem.

We prove (i) by (iv). By (iv), observe that the $\F$-subalgebra $B_0$ of $T$ is a simple unital $\F$-algebra. So it has no nonzero proper two-sided ideal. Hence $B_1=\{O\}$ by Lemma \ref{L;properties} (i). (i) thus follows from the definition of $B_1$.
\end{proof}
For further discussion, we need the following two lemmas.
\begin{lem}\label{L;computation}
Let $n\in \mathbb{N}_0$ and $i_m, j_m, \ell_m\in [0,d]$ for every $m\in [0,n]$.
\begin{enumerate}[(i)]
\item [\em (i)] If $\min\bigcup_{m=0}^n\{k_{i_m},k_{\ell_m}\}=1$, then  $\prod_{m=0}^n(E_{i_m}^*A_{j_m}E_{\ell_m}^*)\in \langle\{E_{i_0}^*JE_{\ell_n}^*\}\rangle_\F$.
\item [\em (ii)] If $k_{i_0}=1$ or $k_{\ell_n}=1$, then $\prod_{m=0}^n(E_{i_m}^*A_{j_m}E_{\ell_m}^*)\in \langle\{E_{i_0}^*JE_{\ell_n}^*\}\rangle_\F$.
\end{enumerate}
\end{lem}
\begin{proof}
We may assume further that $\prod_{m=0}^n(E_{i_m}^*A_{j_m}E_{\ell_m}^*)\neq O$. For (i), note that there is $q\in [0,n]$ such that $\min\{k_{i_q}, k_{\ell_q}\}=1$. Moreover, $E_{i_q}^*A_{j_q}E_{\ell_q}^*\neq O$ by \eqref{Eq;preliminary6} and the fact $\prod_{m=0}^n(E_{i_m}^*A_{j_m}E_{\ell_m}^*)\neq O$. We thus have $E_{i_q}^*A_{j_q}E_{\ell_q}^*=E_{i_q}^*JE_{\ell_q}^*$ by Lemma \ref{L;generalresults} (ii). (i) is shown by combining the equality $E_{i_q}^*A_{j_q}E_{\ell_q}^*=E_{i_q}^*JE_{\ell_q}^*$, \eqref{Eq;preliminary6}, \eqref{Eq;preliminary3}, \eqref{Eq;preliminary1}, and Lemma \ref{L;generalresults} (i). As (ii) is a special case of (i), (ii) is proved by (i).
\end{proof}
\begin{lem}\label{L;rightthin}
If $Z\in \mathrm{Ann}_T(W_0)$, then $ZE_i^*=O$ for every $R_i\in O_\vartheta(S)$.
\end{lem}
\begin{proof}
Suppose that there exists $R_j\in O_\vartheta(S)$ such that $ZE_j^*\neq O$. We thus deduce that $O\neq ZE_j^*=IZE_j^*=\sum_{i=0}^dE_i^*ZE_j^*$ by \eqref{Eq;preliminary2}. So there exists $\ell\in [0,d]$ such that $E_\ell^*ZE_j^*\neq O$. Since $k_j=1$, $Z\in \mathrm{Ann}_T(W_0)$, and $\mathrm{Ann}_T(W_0)$ is a two-sided ideal of $T$, by \eqref{Eq;preliminary3}, Lemmas \ref{L;generalresults} (iii), \ref{L;computation} (ii), we thus have $E_\ell^*ZE_j^*=c_{\ell jZ}E_\ell^*JE_j^*\in\mathrm{Ann}_T(W_0)$, where $0\neq c_{\ell jZ}\in \F$. Since $k_j=1$ and $E_j^*\mathbf{1}\in W_0$ by the definition of $W_0$, by \eqref{Eq;preliminary3} and \eqref{Eq;preliminary5}, we thus have $E_\ell^*ZE_j^*E_j^*\mathbf{1}=c_{\ell jZ}E_\ell^*JE_j^*\mathbf{1}=c_{\ell jZ}E_\ell^*\mathbf{1}\neq \mathbf{0}$, which contradicts the definition of $\mathrm{Ann}_T(W_0)$. The desired lemma thus follows.
\end{proof}
The following lemma includes some characterizations of the $p'$-valenced schemes.
\begin{lem}\label{L;characterization2}
The following statements are equivalent:
\begin{enumerate}[(i)]
\item [\em (i)] $S$ is a $p'$-valenced scheme;
\item [\em (ii)] If $Z\in \mathrm{Ann}_T(W_0)$, then $E_i^*Z=ZE_i^*=O$ for every $R_i\in O_\vartheta(S)$;
\item [\em (iii)]If $\tilde{Z}\in \mathrm{Rad}(T)$, then $E_i^*\tilde{Z}=\tilde{Z}E_i^*=O$ for every $R_i\in O_\vartheta(S)$.
\end{enumerate}
\end{lem}
\begin{proof}
We prove (ii) by (i). By Lemma \ref{L;rightthin}, it is enough to check that $E_i^*Z=O$ for every $R_i\in O_\vartheta(S)$. Suppose that there is $R_j\in O_\vartheta(S)$ such that $E_j^*Z\neq O$. We thus have $O\neq E_j^*Z=E_j^*ZI=\sum_{i=0}^dE_j^*ZE_i^*$ by \eqref{Eq;preliminary2}. So there exists $\ell\in [0,d]$ such that $E_j^*ZE_\ell^*\neq O$. Since $k_j=1$, $Z\in \mathrm{Ann}_T(W_0)$, and $\mathrm{Ann}_T(W_0)$ is a two-sided ideal of $T$, by \eqref{Eq;preliminary3}, Lemmas \ref{L;generalresults} (iii), \ref{L;computation} (ii), we thus have $E_j^*ZE_\ell^*=c_{j\ell Z}E_j^*JE_\ell^*\in \mathrm{Ann}_T(W_0)$, where $0\neq c_{j\ell Z}\in \F$. Since $E_\ell^*\mathbf{1}\in W_0$ by the definition of $W_0$, according to \eqref{Eq;preliminary3} and \eqref{Eq;preliminary5}, we thus have $E_j^*ZE_\ell^*E_\ell^*\mathbf{1}=c_{j\ell Z}E_j^*JE_\ell^*\mathbf{1}=c_{j\ell Z}\overline{k_\ell}E_j^*\mathbf{1}$. By (i), notice that $p\nmid k_\ell$ and
$E_j^*ZE_\ell^*E_\ell^*\mathbf{1}=c_{j\ell Z}\overline{k_\ell}E_j^*\mathbf{1}\neq \mathbf{0}$. The inequality $E_j^*ZE_\ell^*E_\ell^*\mathbf{1}\neq \mathbf{0}$ contradicts the definition of $\mathrm{Ann}_T(W_0)$. (ii) thus follows.

We prove (iii) by (i). Observe that $\mathrm{Rad}(T)\subseteq \mathrm{Ann}_T(W_0)$ by (i) and Lemma \ref{L;primarymoduleproperty} (iv). Observe that (ii) holds since (i) holds. (iii) thus follows from the containment $\mathrm{Rad}(T)\subseteq \mathrm{Ann}_T(W_0)$ and (ii).

We prove (i) by (ii) or (iii). Suppose that (i) does not hold. There exists $a\in [0,d]$ such that $p\mid k_a$. Notice that $O\neq E_0^*E_0^*JE_a^*=E_0^*JE_a^*\in \mathrm{Rad}(T)$ by combining \eqref{Eq;preliminary3}, \eqref{Eq;preliminary4}, the definition of $B_1$, and Lemma \ref{L;properties} (iii). By \eqref{Eq;preliminary3} and \eqref{Eq;preliminary5}, $E_0^*JE_a^*E_b^*\mathbf{1}=\mathbf{0}$ for every $b\in [0,d]$. So $E_0^*JE_a^*\in \mathrm{Ann}_T(W_0)$ by the definitions of $W_0$ and $\mathrm{Ann}_T(W_0)$. Hence we have $O\neq E_0^*E_0^*JE_a^*=E_0^*JE_a^*\in \mathrm{Rad}(T)\cap \mathrm{Ann}_T(W_0)$, which contradicts (ii) and (iii) as $R_0\in O_\vartheta(S)$. Therefore (i) follows if (ii) or (iii) holds.
\end{proof}
\begin{rem}\label{R;remark7}
By \cite[Theorem 3.4]{Han1}, the equality $\mathrm{Rad}(T)=\{O\}$ holds only if $S$ is a $p'$-valenced scheme. Since $\mathrm{Rad}(T)=\{O\}$ implies that Lemma \ref{L;characterization2} (iii) holds, this result can also be verified by Lemma \ref{L;characterization2}. In general, $\mathrm{Rad}(T)$ may not be $\{O\}$ even if $S$ is a $p'$-valenced scheme (see \cite[5.1]{Han1}).
\end{rem}
\begin{eg}\label{E;lemma4.9}
In general, notice that $\mathrm{Rad}(T)$ may not equal $\mathrm{Ann}_T(W_0)$ even if $S$ is a $p'$-valenced scheme. Let us illustrate this fact by a counterexample. Assume that $p>2$ and $S$ is the scheme of order $5$, No. $2$ in \cite{HM}. Observe that $S=\{R_0, R_1, R_2\}$, where $k_0=1$ and $k_1=k_2=2$. Hence $S$ is a $p'$-valenced scheme. By computation, $O\neq E_1^*A_1E_2^*-E_1^*A_2E_2^*\in \mathrm{Ann}_T(W_0)$. However, $\mathrm{Rad}(T)=\{O\}$ by \cite[Theorem B]{Jiang}.
\end{eg}
The following lemma describes a characterization of the $p'$-valenced schemes.
\begin{lem}\label{L;characterization3}
The following statements are equivalent:
\begin{enumerate}[(i)]
\item [\em (i)] $S$ is a $p'$-valenced scheme;
\item [\em (ii)] For every decomposition of $_{T}T$ into a direct sum of indecomposable $T$-modules, there exist exactly $d+1$ indecomposable direct summands isomorphic to $W_0$ as $T$-modules.
\end{enumerate}
\end{lem}
\begin{proof}
We prove (ii) by (i). By (i) and Lemma \ref{L;characterization1} (iii), there is a two-sided ideal $D$ of $T$ such that $T$ is a direct sum of $B_0$ and $D$. Hence $D$ is a $T$-module under the left multiplication action of $T$. Use $_{T}D$ to denote this $T$-module. Hence $_{T}T={_{T}B_0}\oplus {_{T}D}$. Notice that $W_0$ is not isomorphic to a direct summand of $_{T}D$ for every decomposition of $_{T}D$ into a direct sum of $T$-modules. Otherwise, we suppose that there indeed exist $T$-submodules $M$ and $N$ of $_{T}D$ such that $_{T}D=M\oplus N$ and $M\cong W_0$ as $T$-modules. By (i) and Lemma \ref{L;characterization1} (ii), the $\F$-subalgebra $B_0$ of $T$ is unital. Moreover, its identity element $f_{B_0}$ is a central element of $T$. Since $T$ is a direct sum of the two-sided ideals $B_0$ and $D$, notice that $I-f_{B_0}$ is also the identity element of the $\F$-subalgebra $D$ of $T$. By Lemma \ref{L;properties} (v) and the definition of $W_0$, we thus deduce that $f_{B_0}E_i^*\mathbf{1}=E_i^*\mathbf{1}$ for every $i\in [0,d]$. As $_{T}D=M\oplus N$ and $M\cong W_0$ as $T$-modules, by the definition of $W_0$ again, we can also deduce that $(I-f_{B_0})E_i^*\mathbf{1}=E_i^*\mathbf{1}$ for every $i\in [0,d]$. Since $f_{B_0}(I-f_{B_0})=O$, $\mathbf{0}=f_{B_0}(I-f_{B_0})E_i^*\mathbf{1}=E_i^*\mathbf{1}\neq \mathbf{0}$ for every $i\in [0,d]$, which is a contradiction. Since $_{T}T={_{T}B_0}\oplus {_{T}D}$ and Lemmas \ref{L;properties} (v), \ref{L;primarymoduleproperty} (ii) hold, we thus observe that there exists a decomposition of $_{T}T$ into a direct sum of indecomposable $T$-modules such that exactly $d+1$ indecomposable direct summands are isomorphic to $W_0$ as $T$-modules. (ii) thus follows from the Krull-Schmidt Theorem.

For any given $T$-submodule $U$ of $_{T}T$, claim that $U\subseteq B_0$ if $U\cong W_0$ as $T$-modules.
Assume that $U\cong W_0$ as $T$-modules. Let $\phi$ denote a $T$-isomorphism from $W_0$ to $U$. Notice that $W_0$ has an $\F$-basis $\{E_i^*\mathbf{1}: i\in [0,d]\}$ by the definition of $W_0$ and Lemma \ref{L;primemodule} (i). Since $\phi$ is a $T$-isomorphism and $\mathbf{1}=\sum_{i=0}^dE_i^*\mathbf{1}\in W_0$ by \eqref{Eq;preliminary2}, we thus get that $\{E_i^*\phi(\mathbf{1}): i\in [0,d]\}$ is an $\F$-basis of $U$. Since $k_0=1$, $\phi(\mathbf{1})\in U$, and $U\subseteq {_{T}T}$, $E_0^*\phi(\mathbf{1})\in \langle\{E_0^*JE_i^*: i\in [0,d]\}\rangle_\F$ by combining \eqref{Eq;preliminary3}, Lemmas \ref{L;generalresults} (iii), and \ref{L;computation} (ii). Hence $E_0^*\phi(\mathbf{1})\in B_0$ by the definition of $B_0$. As $\phi$ is a $T$-isomorphism and $k_0=1$, by \eqref{Eq;preliminary5}, notice that $E_i^*\phi(\mathbf{1})=\phi(E_i^*\mathbf{1})=\phi(E_i^*JE_0^*\mathbf{1})=E_i^*JE_0^*\phi(\mathbf{1})$ for every $i\in [0,d]$. As $E_0^*\phi(\mathbf{1})\in B_0$ and $B_0$ is a two-sided ideal of $T$, we thus deduce that $E_i^*\phi(\mathbf{1})\in B_0$ for every $i\in [0,d]$. The desired claim thus follows as $B_0$ contains an $\F$-basis of $U$.

We prove (i) by (ii). By (ii), there exist $T$-submodules $V$ and $W$ of $_{T}T$ such that $_{T}T=V\oplus W$ and $V$ is isomorphic to a direct sum of $d+1$ copies of $W_0$ as $T$-modules. Therefore $V\subseteq B_0$ by the Krull-Schmidt Theorem and the proven claim. Moreover, as $\dim_\F W_0=d+1$ and $\dim_\F B_0=(d+1)^2$, we also observe that $\dim_\F V=\dim_\F B_0$. Therefore we have $V=B_0$ and $_{T}T=B_0\oplus W$. So there exist $f_V\in B_0$ and $f_W\in W$ such that $I=f_V+f_W$. As $B_0$ is a two-sided ideal of $T$ and $W$ is a $T$-submodule of $_{T}T$, notice that $Zf_W\in B_0\cap W$ for every $Z\in B_0$. As $_{T}T=B_0\oplus W$, $Zf_W=O$ for every $Z\in B_0$. Hence $Zf_V=Z$ for every $Z\in B_0$. Suppose that (i) does not hold. Then there exists $j\in [0,d]$ such that $p\mid k_j$. Since $f_V\in B_0$, by the definition of $B_0$, notice that $f_V\in \langle\{E_i^*JE_\ell^*: i,\ell\in [0,d]\}\rangle_\F$. Since $E_j^*JE_j^*\in B_0$, we thus deduce that $O=E_j^*JE_j^*f_V=E_j^*JE_j^*\neq O$ by combining \eqref{Eq;preliminary3}, \eqref{Eq;preliminary4}, \eqref{Eq;preliminary5}, and the known result that $Zf_V=Z$ for every $Z\in B_0$. So we have a contradiction. (i) thus follows.
\end{proof}
For further discussion, we recall the following definition and present a lemma.
\begin{defn}\label{D;dualmodule}
Let $U$ be a $T$-module. Let $\mathrm{Hom}_\F(U, \F)$ denote the $\F$-linear space generated by all linear functionals from $U$ to $\F$. By Lemma \ref{L;generalresults} (iii), \eqref{Eq;preliminary6}, and \eqref{Eq;preliminary1}, $Z^t\in T$ for every $Z\in T$. Let $T$ act on $\mathrm{Hom}_\F(U, \F)$ by setting $(Z\psi)(\hat{u})=\psi(Z^t\hat{u})$ for any $Z\in T$, $\psi\in \mathrm{Hom}_\F(U, \F)$, and $\hat{u}\in U$. So $\mathrm{Hom}_\F(U,\F)$ is a $T$-module under the defined action of $T$. Call this $T$-module the contragredient $T$-module of $U$. Let $U^{\circ}$ denote the contragredient $T$-module of $U$. Notice that $\dim_\F U^{\circ}=\dim_\F U$. Call $U$ a self-contragredient $T$-module if $U\cong U^\circ$ as $T$-modules.
\end{defn}
\begin{lem}\label{L;self-contragredient}
If $n\in \mathbb{N}_0$, $Q_n\neq \varnothing$, and $C\in Q_n$, then $Irr_n(C)$ is a self-contragredient $T$-module.
\end{lem}
\begin{proof}
Note that $C\subseteq S_n$ by Lemma \ref{L;primarymoduleadditionalproperty} (i). By the definition of $S_n$, for every $h\in S_n$, there is $q_h\in \mathbb{N}$ such that $k_h=p^nq_h$ and $p\nmid q_h$. By Notation \ref{N;notation3}, recall that $Irr_n(C)$ has an $\F$-basis $\{E_i^*\mathbf{1}+W_{n+1}:i\in C\}$ of cardinality $|C|$. In particular, observe that $\{\overline{q_i}^{-1}E_i^*\mathbf{1}+W_{n+1}:i\in C\}$ is also an $\F$-basis of $Irr_n(C)$ and $\dim_\F Irr_n(C)^\circ=|C|$. For every $i\in C$, let $\psi_i$ denote the linear functional from $Irr_n(C)$ to $\F$ that sends
$E_j^*\mathbf{1}+W_{n+1}$ to $\delta_{ij}$ for every $j\in C$. Since $\{\overline{q_i}^{-1}E_i^*\mathbf{1}+W_{n+1}:i\in C\}$ is an $\F$-basis of $Irr_n(C)$, notice that $\{\psi_i: i\in C\}$ is an $\F$-basis of $Irr_n(C)^\circ$. Let $\Phi$ be the $\F$-linear isomorphism from $Irr_n(C)$ to $Irr_n(C)^\circ$ that sends $\overline{q_i}^{-1}E_i^*\mathbf{1}+W_{n+1}$ to $\psi_i$ for every $i\in C$. By Definition \ref{D;dualmodule}, it is enough to check that $\Phi$ is a $T$-isomorphism. Let $a,b,c\in [0,d]$. We list three cases to prove that $\Phi$ preserves the action of $E_a^*A_bE_c^*$.
\begin{enumerate}[\text{Case} 1:]
\item $a\notin C$.
\end{enumerate}
As $a\notin C$ and $\{E_i^*\mathbf{1}+W_{n+1}: i\in C\}$ is an $\F$-basis of $Irr_n(C)$, by \eqref{Eq;preliminary3}, $E_a^*A_bE_c^*\psi_i$ is the zero element of $Irr_n(C)^\circ$ for every $i\in C$. Notice that $E_a^*A_bE_c^*(\overline{q_i}^{-1}E_i^*\mathbf{1}+W_{n+1})$ equals $\mathbf{0}+W_{n+1}$ for every $i\in C$. Otherwise, suppose that there is $\ell\in C$ such that  $E_a^*A_bE_c^*(\overline{q_\ell}^{-1}E_\ell^*\mathbf{1}+W_{n+1})\neq\mathbf{0}+W_{n+1}$. Since $Irr_n(C)$ is a $T$-submodule of $W_n/W_{n+1}$ with an $\F$-basis $\{E_i^*\mathbf{1}+W_{n+1}: i\in C\}$, by combining \eqref{Eq;preliminary3}, Lemmas \ref{L;generalresults} (i), \ref{L;primemodule} (i), $a\in C$, which is absurd as $a\notin C$. Hence $\Phi$ preserves the action of $E_a^*A_bE_c^*$ by the definition of $\Phi$ and the fact that $\{\overline{q_i}^{-1}E_i^*\mathbf{1}+W_{n+1}: i\in C\}$ is an $\F$-basis of $Irr_n(C)$.
\begin{enumerate}[\text{Case} 2:]
\item $c\notin C$.
\end{enumerate}
As $c\notin C$, by \eqref{Eq;preliminary3},
$E_a^*A_bE_c^*(\overline{q_i}^{-1}E_i^*\mathbf{1}+W_{n+1})=\mathbf{0}+W_{n+1}$ for every $i\in C$. Observe that $E_c^*A_{b'}E_a^*(E_i^*\mathbf{1}+W_{n+1})=\mathbf{0}+W_{n+1}$ for every $i\in C$. Otherwise, suppose that there exists $u\in C$ such that $E_c^*A_{b'}E_a^*(E_u^*\mathbf{1}+W_{n+1})\neq\mathbf{0}+W_{n+1}$. Since $Irr_n(C)$ is a $T$-submodule of $W_n/W_{n+1}$ with an $\F$-basis $\{E_i^*\mathbf{1}+W_{n+1}: i\in C\}$, by combining \eqref{Eq;preliminary3}, Lemmas \ref{L;generalresults} (i), \ref{L;primemodule} (i), $c\in C$, which is absurd as $c\notin C$. As $\{E_i^*\mathbf{1}+W_{n+1}: i\in C\}$ is an $\F$-basis of $Irr_n(C)$ and \eqref{Eq;preliminary1} holds, we thus notice that $E_a^*A_bE_c^*\psi_i$ is the zero element of $Irr_n(C)^{\circ}$ for every $i\in C$. So $\Phi$ preserves the action of $E_a^*A_bE_c^*$ by the definition of $\Phi$ and the fact that $\{\overline{q_i}^{-1}E_i^*\mathbf{1}+W_{n+1}: i\in C\}$ is an $\F$-basis of $Irr_n(C)$.
\begin{enumerate}[\text{Case} 3:]
\item $a,c\in C$.
\end{enumerate}
Let $v\in C$. Since $k_a=p^nq_a$ and $k_v=p^nq_v$, notice that $q_ap_{vb'}^a=q_vp_{ab}^v$ by Lemma \ref{L;trigular}. As $a\in C$, by combining \eqref{Eq;preliminary3}, Lemma \ref{L;generalresults} (i), and the definition of $\Phi$, we thus have
\begin{align*}
\Phi(E_a^*A_bE_c^*(\overline{q_v}^{-1}E_v^*\mathbf{1}+W_{n+1}))&=
\Phi(\delta_{cv}\overline{q_v}^{-1}\overline{p_{vb'}^a}(E_a^*\mathbf{1}+W_{n+1}))\\
&=\Phi(\delta_{cv}\overline{q_a}^{-1}\overline{p_{ab}^v}(E_a^*\mathbf{1}+W_{n+1}))
=\delta_{cv}\overline{p_{ab}^v}\psi_a.
\end{align*}
By the definition of $\Phi$ again, we also have
$E_a^*A_bE_c^*\Phi(\overline{q_v}^{-1}E_v^*\mathbf{1}+W_{n+1})=E_a^*A_bE_c^*\psi_v
$. Let $w\in C$. As $c\in C$, by combining \eqref{Eq;preliminary3}, \eqref{Eq;preliminary1}, and Lemma \ref{L;generalresults} (i), notice that
\begin{align*} \delta_{cv}\overline{p_{ab}^v}\psi_a(E_w^*\mathbf{1}+W_{n+1})=\delta_{cv}\delta_{aw}\overline{p_{ab}^v}=
\delta_{cv}\delta_{aw}\overline{p_{wb}^c}&=\psi_v(E_c^*A_{b'}E_a^*(E_w^*\mathbf{1}+W_{n+1}))\\
&=E_a^*A_bE_c^*\psi_v(E_w^*\mathbf{1}+W_{n+1}),
\end{align*}
which implies that $\delta_{cv}\overline{p_{ab}^v}\psi_a=E_a^*A_bE_c^*\psi_v$ since $\{E_i^*\mathbf{1}+W_{n+1}: i\in C\}$ is an $\F$-basis of $Irr_n(C)$ and $w$ is chosen from $C$ arbitrarily. We thus deduce that
$$\Phi(E_a^*A_bE_c^*(\overline{q_v}^{-1}E_v^*\mathbf{1}+W_{n+1}))=\delta_{cv}\overline{p_{ab}^v}\psi_a=
E_a^*A_bE_c^*\psi_v=\!\!E_a^*A_bE_c^*\Phi(\overline{q_v}^{-1}E_v^*\mathbf{1}+W_{n+1}).$$
Therefore $\Phi$ preserves the action of $E_a^*A_bE_c^*$ since $\{\overline{q_i}^{-1}E_i^*\mathbf{1}+W_{n+1}: i\in C\}$ is an $\F$-basis of $Irr_n(C)$ and $v$ is chosen from $C$ arbitrarily.

By Cases 1, 2, 3, $\Phi$ preserves the action of $E_a^*A_bE_c^*$. As $a, b,c$ are chosen from $[0,d]$ arbitrarily, by Lemma \ref{L;generalresults} (iii) and \eqref{Eq;preliminary6}, we thus obtain that $\Phi$ is a $T$-isomorphism. The desired lemma thus follows.
\end{proof}
The following lemma gives a characterization of the $p'$-valenced schemes.
\begin{lem}\label{L;characterization4}
The following statements are equivalent:
\begin{enumerate}[(i)]
\item [\em (i)] $S$ is a $p'$-valenced scheme;
\item [\em (ii)] $W_0$ is a self-contragredient $T$-module.
\end{enumerate}
\end{lem}
\begin{proof}
We prove (ii) by (i). By (i), note that $S_0=[0,d]$. By combining Lemmas \ref{L;primarymoduleproperty} (iii), (iv), \ref{L;decomposition} (ii), Remark \ref{R;remark1}, $W_0=Irr_0([0,d])$. So (ii) is from Lemma \ref{L;self-contragredient}.

We prove (i) by (ii). Since $\dim_\F W_0^{\circ}=\dim_\F W_0=d+1$, let $\{\psi_i: i\in [0,d]\}$ be an $\F$-basis of $W_0^{\circ}$. By (ii), let $\Phi$ denote a $T$-isomorphism from $W_0$ to $W_0^{\circ}$. So there exist $c_0, c_1, \ldots, c_d\in \F$ such that $\Phi(E_0^*\mathbf{1})=\sum_{i=0}^dc_i\psi_i$. Suppose that (i) does not hold. Then there is $j\in [0,d]$ such that $p\mid k_j$. Notice that $E_j^*J\in T$ by \eqref{Eq;preliminary2}. By \eqref{Eq;preliminary3} and \eqref{Eq;preliminary5}, $JE_j^*E_i^*\mathbf{1}=\mathbf{0}$ for every $i\in [0,d]$. As $\{E_i^*\mathbf{1}: i\in [0,d]\}$ is an $\F$-basis of $W_0$, by \eqref{Eq;preliminary1}, we thus get that $E_j^*J\psi_i$ is the zero element of $W_0^{\circ}$ for every $i\in [0,d]$. Since $\Phi$ is a $T$-isomorphism, $\Phi(E_j^*JE_0^*\mathbf{1})=E_j^*J\Phi(E_0^*\mathbf{1})=\sum_{i=0}^dc_iE_j^*J\psi_i$, which implies that $E_j^*JE_0^*\mathbf{1}=\mathbf{0}$. This is absurd as $E_j^*JE_0^*\mathbf{1}=E_j^*\mathbf{1}\neq\mathbf{0}$ by \eqref{Eq;preliminary5}. (i) thus follows.
\end{proof}
For our purpose, we also introduce the following notation and list four lemmas.
\begin{nota}\label{N;notation5}
Let $U$ denote a $T$-module. Define $E_i^*U=\{E_i^*\hat{u}: \hat{u}\in U\}$ for every $i\in [0,d]$. For every $i\in [0,d]$, note that $E_i^*U$ is an $\F$-linear subspace of $U$. Therefore $\dim_\F U\geq \dim_\F E_i^*U$ for every $i\in [0,d]$.
\end{nota}
\begin{lem}\label{L;Eric1}
The following statements are equivalent:
\begin{enumerate}[(i)]
\item [\em (i)] $U\cong W_0/W_1$ as $T$-modules;
\item [\em (ii)]$U$ is an irreducible $T$-module satisfying $\dim_\F E_i^*U>0$ for some $R_i\in O_\vartheta(S)$.
\end{enumerate}
\end{lem}
\begin{proof}
We prove (ii) by (i). By (i) and Lemma \ref{L;primarymoduleproperty} (iii), $U$ is an irreducible $T$-module. Note that $\{E_j^*\mathbf{1}+W_1: j\in S_0\}$ is an $\F$-basis of $W_0/W_1$ by Lemma \ref{L;submoduleseries} (ii). As $0\in S_0$, we thus have $\dim_\F E_0^*(W_0/W_1)=1>0$ by \eqref{Eq;preliminary3} and the definition of $E_0^*(W_0/W_1)$. So $\dim_\F E_0^*U=1>0$ by (i) and the definition of $E_0^*U$. (ii) follows as $R_0\in O_\vartheta(S)$.

We prove (i) by (ii). As $U$ is an irreducible $T$-module, $U$ is a $T$-module generated by a single element. So $_{T}T/V\cong U$ as $T$-modules for some $T$-submodule $V$ of $_{T}T$. As $\dim_\F E_i^*U>0$ for some $R_i\in O_\vartheta(S)$ and $_{T}T/V\cong U$ as $T$-modules, according to the definition of $E_i^*U$, there exists $Z\in T$ such that $E_i^*(Z+V)=E_i^*Z+V\neq O+V$. So
$E_i^*Z+V=E_i^*ZI+V=\sum_{j=0}^dE_i^*ZE_j^*+V\neq O+V$ by \eqref{Eq;preliminary2}. Hence there exists $\ell\in [0,d]$ such that $E_i^*ZE_\ell^*+V\neq O+V$. Notice that $k_i=1$ as $R_i\in O_\vartheta(S)$. We thus have $E_i^*ZE_\ell^*+V=c_{i\ell Z} E_i^*JE_\ell^*+V\neq O+V$ and $0\neq c_{i\ell Z}\in \F$ by combining \eqref{Eq;preliminary3}, Lemmas \ref{L;generalresults} (iii), \ref{L;computation} (ii). In particular, $\langle\{ E_j^*JE_\ell^*+V: j\in [0,d]\}\rangle_\F\neq \{O+V\}$.

We claim that $\langle\{ E_j^*JE_\ell^*+V: j\in [0,d]\}\rangle_\F={_{T}}T/V$. As $_{T}T/V\cong U$ as $T$-modules and $U$ is an irreducible $T$-module, notice that $_{T}T/V$ is an irreducible $T$-module. As we also have $\langle\{ E_j^*JE_\ell^*+V: j\in [0,d]\}\rangle_\F\neq \{O+V\}$, the desired claim follows if we verify that $\langle\{ E_j^*JE_\ell^*+V: j\in [0,d]\}\rangle_\F$ is a $T$-submodule of $_{T}T/V$. Let $a,b,c\in [0,d]$. For every $h\in [0,d]$, notice that $E_a^*A_bE_c^*(E_h^*JE_\ell^*+V)=\delta_{ch}\overline{p_{hb'}^a}E_a^*JE_\ell^*+V$ by \eqref{Eq;preliminary3} and Lemma \ref{L;generalresults} (i). So $E_a^*A_bE_c^*(E_h^*JE_\ell^*+V)\in \langle\{ E_j^*JE_\ell^*+V: j\in [0,d]\}\rangle_\F$ for every $h\in [0,d]$. As $\langle\{ E_j^*JE_\ell^*+V: j\in [0,d]\}\rangle_\F$ is an $\F$-linear space and $a,b,c$ are chosen from $[0,d]$ arbitrarily, we thus get that $\langle\{ E_j^*JE_\ell^*+V: j\in [0,d]\}\rangle_\F$ is a $T$-submodule of $_{T}T/V$ by Lemma \ref{L;generalresults} (iii) and \eqref{Eq;preliminary6}. The desired claim thus follows.

By the definition of $M_\ell$ and the proven claim, there exists an obvious surjective $T$-homomorphism from $M_\ell$ to $_{T}T/V$. So there is a surjective $T$-homomorphism from $W_0$ to $_{T}T/V$ by Lemma \ref{L;properties} (iv). As $U$ is an irreducible $T$-module and $_{T}T/V\cong U$ as $T$-modules, by Lemma \ref{L;primarymoduleproperty} (i), $W_0/W_1\cong{ _{T}T}/V\cong U$ as $T$-modules. (i) is proved.
\end{proof}
\begin{lem}\label{L;Eric2}
The following statements are equivalent:
\begin{enumerate}[(i)]
\item [\em (i)] $S$ is a $p'$-valenced scheme;
\item [\em (ii)] $U$ is an irreducible $T$-module satisfying $\dim_\F E_i^*U>0$ for some $R_i\in O_\vartheta(S)$ if and only if $U\cong W_0$ as $T$-modules.
\end{enumerate}
\end{lem}
\begin{proof}
We prove (ii) by (i). According to (i) and Lemma \ref{L;primarymoduleproperty} (iv), observe that $W_0$ is an irreducible $T$-module. So $W_1=\{\mathbf{0}\}$ by Lemma \ref{L;primarymoduleproperty} (i). Therefore (ii) follows from Lemma \ref{L;Eric1}. We prove (i) by (ii). By (ii) and Lemma \ref{L;Eric1}, we have $W_0\cong W_0/W_1$ as $T$-modules. So $W_0$ is an irreducible $T$-module by Lemma \ref{L;primarymoduleproperty} (iii). Therefore (i) follows from \ref{L;primarymoduleproperty} (iv).
\end{proof}
\begin{lem}\label{L;Eric3}
If $S$ is a $p'$-valenced scheme, the following statements are equivalent:
\begin{enumerate}[(i)]
\item [\em (i)] $U\cong W_0$ as $T$-modules;
\item [\em (ii)] $U$ is a $T$-module satisfying $\dim_\F U=\!d+1\!\geq\!\dim_\F E_i^*U\!\!>0$ for some $R_i\in O_\vartheta(S)$.
\end{enumerate}
\end{lem}
\begin{proof}
We prove (ii) by (i). We recall that $W_0$ has an $\F$-basis $\{E_j^*\mathbf{1}: j\in [0,d]\}$ and $\dim_\F W_0=d+1$. By \eqref{Eq;preliminary3} and the definition of $E_0^*W_0$, note that $\dim_\F E_0^*W_0=1>0$. So we have $\dim_\F U=d+1\geq \dim_\F E_0^*U=1>0$ by (i) and the definition of $E_0^*U$. (ii) thus follows as $R_0\in O_\vartheta(S)$.

We prove (i) by (ii). By (ii), notice that $\dim_\F E_i^*U>0$ for some $R_i\in O_\vartheta(S)$. Pick an element of an $\F$-basis of $E_i^*U$. Let $M$ denote the $T$-submodule of $U$ generated by this chosen element. So there exists a $T$-submodule $V$ of $_{T}T$ such that $_{T}T/V\cong M$ as $T$-modules. Since $M$ contains an element of an $\F$-basis of $E_i^*U$ and $_{T}T/V\cong M$ as $T$-modules, according to the definition of $E_i^*U$, notice that there is $Z\in T$ such that $E_i^*(Z+V)=E_i^*Z+V\neq O+V$.  So $E_i^*Z+V=E_i^*ZI+V=\sum_{j=0}^dE_i^*ZE_j^*+V\neq O+V$ by \eqref{Eq;preliminary2}. Hence there exists $\ell\in [0,d]$ such that $E_i^*ZE_\ell^*+V\neq O+V$. Notice that $k_i=1$ as $R_i\in O_\vartheta(S)$. We thus have $E_i^*ZE_\ell^*+V=c_{i\ell Z}E_i^*JE_\ell^*+V\neq O+V$ and $0\neq c_{i\ell Z}\in \F$ by combining \eqref{Eq;preliminary3}, Lemmas \ref{L;generalresults} (iii), \ref{L;computation} (ii). In particular, notice that $O+V\neq E_i^*JE_\ell^*+V\in \langle\{E_j^*JE_\ell^*+V: j\in [0,d]\}\rangle_\F$.

We claim that $\langle\{E_j^*JE_\ell^*+V: j\in [0,d]\}\rangle_\F={_{T}T}/V\cong M=U$. As we have known that $\langle\{E_j^*JE_\ell^*+V: j\in [0,d]\}\rangle_\F\subseteq{_{T}T}/V\cong M\subseteq U$, the desired claim thus follows if we can verify that $\dim_\F\langle\{E_j^*JE_\ell^*+V: j\in [0,d]\}\rangle_\F=\dim_\F U$. We now suppose that $\sum_{j=0}^dc_jE_j^*JE_\ell^*+V=O+V$, where $\bigcup_{j=0}^d\{c_j\}\subseteq \F$ and
$(\bigcup_{j=0}^d\{c_j\})\cap(\F\setminus\{0\})\neq \varnothing$. Therefore there exists $a\in [0,d]$ such that $c_a\neq 0$. According to \eqref{Eq;preliminary3}, observe that $c_a E_a^*JE_\ell^*+V=E_a^*(\sum_{j=0}^dc_j E_j^*JE_\ell^*+V)=E_a^*(O+V)=O+V$, which implies that $E_a^*JE_\ell^*+V=O+V$ as $c_a\neq 0$. As $S$ is a $p'$-valenced scheme, notice that $p\nmid k_a$. By \eqref{Eq;preliminary3} and \eqref{Eq;preliminary5}, we thus have $\overline{k_a}E_i^*JE_\ell^*+V=E_i^*JE_a^*(E_a^*JE_\ell^*+V)=O+V$ and $E_i^*JE_\ell^*+V=O+V$, which contradicts the inequality $E_i^*JE_\ell^*+V\neq O+V$. Hence we deduce that $\dim_\F \langle\{E_j^*JE_\ell^*+V: j\in [0,d]\}\rangle_\F=d+1$. The desired claim thus follows as $\dim_\F U=d+1$ by (ii).

By the definition of $M_\ell$ and the proven claim, there exists an obvious surjective $T$-homomorphism from $M_\ell$ to $_{T}T/V$. So there exists a surjective $T$-homomorphism from $W_0$ to $_{T}T/V$ by Lemma \ref{L;properties} (iv). According to the proven claim and (ii), notice that $_{T}T/V\cong U$ as $T$-modules and $\dim_\F U=\dim_\F W_0=d+1$. (i) thus follows.
\end{proof}
\begin{lem}\label{L;Eric4}
Assume that the following statements are equivalent:
\begin{enumerate}[(i)]
\item [\em (i)] $U\cong W_0$ as $T$-modules;
\item [\em (ii)]$U$ is a $T$-module satisfying $\dim_\F U=\!d+1\!\geq\!\dim_\F E_i^*U\!\!>0$ for some $R_i\in O_\vartheta(S)$.
\end{enumerate}
Then $S$ is a $p'$-valenced scheme.
\end{lem}
\begin{proof}
Recall that $W_0$ has an $\F$-basis $\{E_j^*\mathbf{1}: j\in [0,d]\}$. For every $j\in [0,d]$, let $\psi_j$ denote the linear functional from $W_0$ to $\F$ that sends $E_\ell^*\mathbf{1}$ to $\delta_{j\ell}$ for every $\ell\in [0,d]$. Notice that $\{\psi_j: j\in [0,d]\}$ is an $\F$-basis of $W_0^{\circ}$ and $\dim_\F W_0^{\circ}=d+1$. According to the definition of $E_0^*W_0^{\circ}$, \eqref{Eq;preliminary1}, and \eqref{Eq;preliminary3}, notice that $E_0^*W_0^{\circ}=\langle\{\psi_0\}\rangle_\F$. So we have $\dim_\F W_0^{\circ}=d+1\geq\dim_\F E_0^*W_0^{\circ}=1>0$. Since $R_0\in O_\vartheta(S)$ and (ii) implies (i), $W_0$ is a self-contragredient $T$-module. So $S$ is a $p'$-valenced scheme by Lemma \ref{L;characterization4}.
\end{proof}
We are now ready to present the main result of this section.
\begin{thm}\label{T;characterization}
The following statements are equivalent:
\begin{enumerate}[(i)]
\item [\em (i)] $S$ is a $p'$-valenced scheme;
\item [\em (ii)] The $\F$-subalgebra $B_0$ of $T$ is unital. Its identity element is central in $T$;
\item [\em (iii)] There exists a two-sided ideal $D$ of $T$ such that $T$ is a direct sum of $B_0$ and $D$;
\item [\em (iv)] The $\F$-subalgebra $B_0$ of $T$ is isomorphic to a full matrix algebra over a division $\F$-algebra as $\F$-algebras;
\item [\em (v)] If $Z\in \mathrm{Ann}_T(W_0)$, then $E_i^*Z=ZE_i^*=O$ for every $R_i\in O_\vartheta(S)$;
\item [\em (vi)] If $\tilde{Z}\in \mathrm{Rad}(T)$, then $E_i^*\tilde{Z}=\tilde{Z}E_i^*=O$ for every $R_i\in O_\vartheta(S)$;
\item [\em (vii)] For every decomposition of $_{T}T$ into a direct sum of indecomposable $T$-modules, there exist exactly $d+1$ indecomposable direct summands isomorphic to $W_0$ as $T$-modules;
\item [\em (viii)] $W_0$ is an irreducible $T$-module;
\item [\em (ix)] $W_0$ is a self-contragredient $T$-module;
\item [\em (x)]$U$ is an irreducible $T$-module satisfying $\dim_\F E_j^*U>0$ for some $R_j\in O_\vartheta(S)$ if and only if $U\cong W_0$ as $T$-modules;
\item [\em (xi)] $U$ is a $T$-module satisfying $\dim_\F U=\!d+1\!\geq\!\dim_\F E_{\tilde{j}}^*U\!\!>0$ for some $R_{\tilde{j}}\in O_\vartheta(S)$ if and only if $U\cong W_0$ as $T$-modules.
\end{enumerate}
\end{thm}
\begin{proof}
(i) is equivalent to all the other listed statements by combining Lemmas \ref{L;characterization1}, \ref{L;characterization2}, \ref{L;characterization3}, \ref{L;primarymoduleproperty} (iv), \ref{L;characterization4}, \ref{L;Eric2}, \ref{L;Eric3}, and \ref{L;Eric4}. Hence the desired theorem follows.
\end{proof}
Let $\F G$ be the group algebra of a finite group $G$ over $\F$. Note that $\F G$ itself is an $\F G$-module under the left multiplication action of $\F G$. Denote this $\F G$-module by $_{\F G}\F G$. It is known that the following statements are equivalent: 
\begin{enumerate}[(i)]
\item The principal block algebra of $\F G$ is a simple unital $\F$-algebra;
\item The Jacobson radical of $\F G$ is the zero ideal;
\item For every decomposition of $_{\F G}\F G$ into a direct sum of indecomposable $\F G$-modules, there exists exactly one indecomposable direct summand isomorphic to the trivial $\F G$-module as $\F G$-modules.
\end{enumerate}
The following corollary tells us that similar statements about $T$ are also equivalent. Its proof comes from the Artin-Wedderburn Theorem and Theorem \ref{T;characterization}.
\begin{cor}\label{C;p'-schemes}
The following statements are equivalent:
\begin{enumerate}[(i)]
\item [\em (i)] The $\F$-subalgebra $B_0$ of $T$ is a simple unital $\F$-algebra;
\item [\em (ii)] If $Z\in \mathrm{Rad}(T)$, then $E_i^*Z=ZE_i^*=O$ for every $R_i\in O_\vartheta(S)$;
\item [\em (iii)] For every decomposition of $_{T}T$ into a direct sum of indecomposable $T$-modules, there exist exactly $d+1$ indecomposable direct summands isomorphic to $W_0$ as $T$-modules.
\end{enumerate}
\end{cor}
We close this note by proposing two questions motivated by our main results.

Every composition factor of $W_0$ is an irreducible self-contragredient $T$-module by Theorem \ref{T;primarymodule} and Lemma \ref{L;self-contragredient}. This fact motivates us to ask the following question.
\begin{ques}\label{Q;question1}
\em Can one determine all irreducible self-contragredient $T$-modules up to isomorphism?
\end{ques}
We call an indecomposable $T$-module a semiprimary $T$-module if this $T$-module satisfies Lemma \ref{L;Eric3} (ii). It is obvious that $W_0$ and $W_0^{\circ}$ are semiprimary $T$-modules. By Theorem \ref{T;characterization}, note that a semiprimary $T$-module may not be isomorphic to $W_0$ as $T$-modules. So the notion of a semiprimary $T$-module generalizes the notion of the primary $T$-module. The definition of a semiprimary $T$-module motivates us to ask the following question.
\begin{ques}\label{Q;question2}
\em Can one determine all semiprimary $T$-modules up to isomorphism?
\end{ques}

\end{document}